\documentclass[a4paper,12pt]{scrartcl}

\usepackage[T1]{fontenc}
\usepackage{amsmath,amssymb}
\usepackage{amsfonts}
\usepackage{graphicx}
\usepackage{color}
\usepackage{algorithm}
\usepackage{algorithmic}
\usepackage{tikz}
\usepackage{multirow}
\usepackage{multicol}
\usepackage{paralist}
\usepackage{pgfplots}

\pgfplotsset{compat=newest}
\pgfplotsset{plot coordinates/math parser=false}

\newtheorem{theorem}{Theorem}
\newtheorem{problem}[theorem]{Problem}

\newtheorem{corollary}[theorem]{Corollary}

\newenvironment{proof}{\noindent {\bf Proof:}}{
 {\hfill{\vrule height 6pt width 6pt depth 0pt}}\newline}
\newtheorem{lemma}[theorem]{Lemma}

\newtheorem{remark}[theorem]{Remark}
\newtheorem{definition}[theorem]{Definition}

\def\lhb#1{\mathcal{L}(#1)} 
\def\rhb#1{\mathcal{R}(#1)}

\def\mean#1{\langle #1 \rangle}
\def\csd{C_{SD}}

 \evensidemargin=\oddsidemargin
 \newdimen\dummy
 \dummy=\oddsidemargin
\newcommand{\Index}{{\mathcal I}}
\newcommand{\argmin}{\mathop{\rm argmin}}

\newcommand{\Ind}{\mathcal{I}}
\newcommand{\R}{\ensuremath{\mathbb{R}}}

\def\plottitlefontsize{\footnotesize}
\def\plotlabelfontsize{\normalsize}
\def\plotlegendfontsize{\scriptsize}
\def\plottickfontsize{\small}
\def\plotlinewidth{1.0pt}

\begin{document}
\title{Alternating Least Squares Tensor Completion in the TT-Format}
\author{
  Lars Grasedyck\thanks{Institut f\"ur Geometrie und Praktische Mathematik, 
    RWTH Aachen, Templergraben 55, 52056 Aachen, Germany. 
    Email: {\tt \{lgr,kluge,kraemer\}@igpm.rwth-aachen.de}. All three authors gratefully acknowledge support by the
    DFG priority programme 1324 under grant GR3179/2-2, the first and last author 
    gratefully acknowledge support by the DFG priority programme 1648 under grant GR3179/3-1.
}
  \and
  Melanie Kluge$^*$
  \and 
  Sebastian Kr\"amer$^*$ \\
  \small{(Accepted for publication in
SIAM Journal on Scientific Computing (SISC))}
}

\maketitle
\sloppy

\begin{abstract}
  We consider the problem of fitting a low rank tensor $A\in\mathbb{R}^{\Index}$,
  $\Index = \{1,\ldots,n\}^{d}$, to a given set of data points 
  $\{M_i\in\mathbb{R}\mid i\in P\}$, $P\subset\Index$.
  The low rank format under consideration is the hierarchical or TT or MPS format. It is 
  characterized by rank bounds $r$ on certain matricizations of the tensor. The number of degrees 
  of freedom is in ${\cal O}(r^2dn)$. 
  For a fixed rank and mode size $n$ we observe that it is possible to reconstruct 
  random (but rank structured) tensors as well as certain discretized multivariate 
  (but rank structured) functions  
  from a number of samples that is in ${\cal O}(\log N)$ for a tensor having $N=n^d$ entries.
  We compare an alternating least squares fit (ALS) to an overrelaxation scheme inspired by the 
  LMaFit method for matrix completion. 
  Both approaches aim at finding a tensor $A$ that fulfils the first order optimality conditions by a 
  nonlinear Gauss-Seidel type solver that consists of an alternating fit cycling through the 
  directions $\mu=1,\ldots,d$. The least squares fit is of complexity ${\cal O}(r^4d\#P)$ per step, 
  whereas each step of ADF is in ${\cal O}(r^2d\#P)$, albeit with a slightly higher number of necessary steps.
  In the numerical experiments we observe robustness of the completion algorithm with respect 
  to noise and good reconstruction capability. Our tests provide evidence that the algorithm is 
  suitable in higher dimension ($>$10) as well as for moderate ranks.  

  Keywords: MPS, Tensor Completion, Tensor Train, TT, Hierarchical Tucker, HT, ALS.\\
  MSC: 15A69, 65F99
\end{abstract}

\section{Introduction}\label{Sec:intro}

We consider the problem of fitting a low rank tensor 
\[A\in\mathbb{R}^{\Index},\quad \Index := \Index_1\times\cdots\times\Index_d,\quad \Index_{\mu}:=\{1,\ldots,n_{\mu}\},\quad
\mu\in D:=\{1,\ldots,d\},\]
to given data points 
\[\{M_i \in\mathbb{R}\mid i\in P\},\quad P\subset\Index,\qquad \#P\ge \sum_{\mu=1}^{d} n_\mu,\]
by minimizing the distance between the given values $(M_i)_{i\in P}$ and approximations
$(A_i)_{i\in P}$:
\[ A = \mathop{\rm argmin}_{\tilde{A}\in T}\sum_{i\in P}(M_i-\tilde{A}_i)^{2}\qquad\qquad \mbox{($T$ being a certain tensor class)}\]
In the class of general dense tensors this is trivial, because the entries of the tensor
are all independent. For sparse tensors this reduces to a simple knapsack problem. 
Our target tensor class is the set of low rank tensors, i.e., we assume that the implicitly given tensor 
$M \in\mathbb{R}^{\Index}$
allows for a low rank approximation
\[ \|M-\tilde{M}\| \le \varepsilon,\qquad \varepsilon\in\mathbb{R}_{\ge 0},\]
where the unknown approximant $\tilde{M}\in\mathbb{R}^{\Index}$ fulfils certain rank bounds that will be
introduced later. In particular we allow $\varepsilon=0$ so that the task is to  
reconstruct the whole tensor $M=\tilde{M}$ in the low rank format. 
This particular case is considered, e.g. in \cite{LiuMuWoYe09,GaReYa11}.

\subsection{Completion versus Sampling}

A tensor fitting problem might arise as follows: the entries $(M_i)_{i\in P}$ could be measurements of a 
multiparameter model such that each index $i\in P$ represents a specific choice of 
$d$ parameters. If the measurements are incomplete or in parts known to be incorrect,
then the goal is to reconstruct all values of $M$ for all parameter combinations  $i\in \Index$ from 
the known values $(M_i)_{i\in P}$ (prior to the assumption that $M$ allows for an approximation in 
the low rank format). It is crucial that the points $P$ are given and we are not free to 
choose them. In case that the points can be chosen freely one after another, the problem simplifies 
drastically and can be approached as in \cite{OsTy10,BaGrKl13} by an adaptive
sampling strategy.
Sometimes one can propose rules on how the entries from $P$ should be chosen, as it is done 
in quasi Monte Carlo methods. This approach is persued in \cite{Kl13} and defines sampling rules
that allow an efficient approximation scheme. Again, this is different and possibly a simpler task
than the tensor completion considered here.

\subsection{Low Rank Tensor Formats}

The class of tensors in which we aim for a completion of the given tensor entries is a low 
rank format. In the case $d=2$ the rank of a tensor coincides with the usual matrix rank, but in 
dimension $d>2$ there are several possibilities to define the rank of a tensor and thus there are 
several data-sparse low rank formats available.

In the CP($k$) format\footnote{CP stands for canonical polyadic, in the literature also called 
CANDECOMP and PARAFAC} or representation
\[
A  = \sum_{\ell=1}^{k} \mathop{\otimes}\limits_{\mu=1}^{d} g_{\mu,\ell},\qquad
A_{i_1,\ldots,i_d} = \sum_{\ell=1}^{k} \prod_{\mu=1}^{d} g_{\mu,\ell}(i_\mu),\qquad g_{\mu,\ell}(i_\mu)\in\mathbb{R},
\]
the tensor completion problem has been considered in \cite{ToBro05,AcDuKoMob11,KriSi13}. The minimal number of summands $k$ by which 
the tensor $A$ can be represented is the \emph{tensor rank} of $A$, but minimality of $k$ is often not
relevant in applications. The CP($k$) format is data sparse in the sense that storing the factors $g_{\mu,\ell}$ amounts to 
${\cal O}(dnk)$ units (real numbers) of storage, as opposed to the $n^d$ units of the full dense and
unstructured tensor $A$. 
This is the reason for the attractivity of the format despite many theoretical and practical 
difficulties \cite{KoBad09}.

In the Tucker format 
\[
A_{i_1,\ldots,i_d} = \sum_{\ell_1=1}^{k_1}\cdots\sum_{\ell_d=1}^{k_d} C_{\ell_1,\ldots,\ell_d}\prod_{\mu=1}^{d} g_{\mu,\ell_\mu}(i_\mu),\qquad g_{\mu,\ell}(i_\mu)\in\mathbb{R},\quad
C \in \mathbb{R}^{k_1\times\cdots\times k_d},\]
tensor completion has been considered in \cite{SiDiLaSu11,KrStVa14,LiuSha13,RaSchSt13}. This format is limited to small dimensions
$d$ since the so-called core tensor $C$ requires $\prod_{\mu=1}^{d}k_\mu$ units of storage. The advantage on
the other hand is that standard matrix approximation techniques can be used by matricizing the tensor. 

The low rank format that we consider lies in between these 
two, combining the benefits of both: the number of degrees of freedom scales linearly with the dimension 
$d$ and the format is based on matricizations such that standard linear algebra tools are applicable.

Here, we put no special assumptions on the data points $P$, except that they are reasonably distributed: 
\begin{definition}[Slices and slice density]\label{def:fibre_density}
  We define the slice density of a point set $\{M_i\in\mathbb{R}\mid i\in P\}, P\subset\Index$, 
  in direction $\mu\in D$ and index {\boldmath$j_\mu$}$\in\Index_{\mu}$ by
  \[c(\mbox{\boldmath$j_\mu$}) := \#\{i\in P\mid i_\mu = \mbox{\boldmath$j_\mu$}\}\]
  The corresponding slice of a tensor $A\in\mathbb{R}^\Index$ is defined by 
  \[ A_{i_\mu = \mbox{\boldmath\scriptsize$j_\mu$}} := 
  \hat{A}\in\mathbb{R}^{\Index_1\times\cdots\times\Index_{\mu-1}\times\Index_{\mu+1}\times\cdots\times\Index_d},\qquad 
  \hat{A}_{i_1,\ldots,i_{\mu-1},i_{\mu+1},\ldots,i_d} := A_{i_1,\ldots,i_{\mu-1},\mbox{\boldmath\scriptsize$j_\mu$},i_{\mu+1},\ldots,i_d}\]
\end{definition}
Depending on the rank parameters $r_\mu$ of $A$ (which in turn depend on the target accuracy of the 
approximation) the slice densities of the set $P$ have to be high enough, i.e. 
\[ c(\mbox{\boldmath$j_\mu$}) \ge \csd r_\mu^2,\qquad \mbox{\boldmath$j_\mu$}\in\Index_{\mu},\quad \mu\in D,\]
for a constant $\csd$, the oversampling factor or overall slice density relative to the rank. 
Note that thereby, the minimal value for $\#P$ increases if any $\#\Index_{\mu}$ does.
If one of the values $c(\mbox{\boldmath$j_\mu$})$ were zero, then this simply means that the slice 
$A_{i_\mu=\mbox{\boldmath\scriptsize$j_\mu$}}$
is undetermined and not observable for any of the low rank formats mentioned above and in the following.
In a minimum norm sense the completed tensor could be set to zero for this slice without any effect on
the rank or approximation in the known points $P$.

The low rank format under consideration is the hierarchical \cite{HaKue09,Gr10} or TT \cite{OsTy09,Os11} or MPS \cite{Whi92,Vi03}
format.
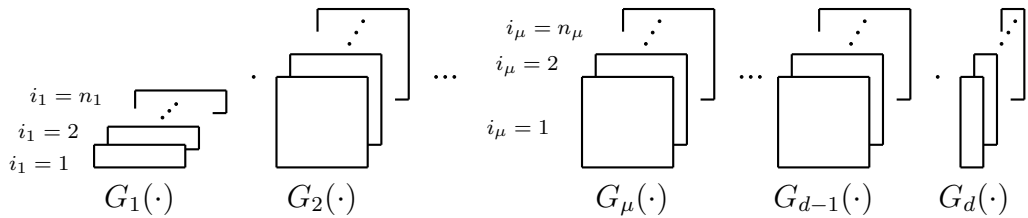
\begin{figure}[htb]
  \begin{center}
    \begin{tikzpicture}[label distance=-1mm,scale = 0.60]
      
      \draw [-, thick] (-4.0,0) -- (-4.0,0.5);
      \draw [-, thick] (-4.0,0) -- (-2.0,0);
      \foreach \x in {1,3}
      \draw [-, thick] (-4.0+ \x*0.3,0.5+\x*0.4) -- (-4.0+ \x*0.3,0.1+\x*0.4);
      \foreach \x in {1,3}
      \draw [-, thick] (-4.0+ \x*0.3+1.7,\x*0.4) -- (-4.0+ \x*0.3+2,\x*0.4);
      \foreach \x in {0,1,3}
      \draw [-, thick] (-4.0+ \x*0.3,0.5+\x*0.4) -- (-4.0+ \x*0.3+2,0.5+\x*0.4);
      \foreach \x in {0,1,3}
      \draw [-, thick] (-4.0+ \x*0.3+2,0.5+\x*0.4) -- (-4.0+ \x*0.3+2,\x*0.4);
      \foreach \x in {0,1,2}
      \node[draw,circle,inner sep=0.3pt,fill] at (-4.0+ 1.5+\x*0.15,1.1+\x*0.15) {};

      \node[below] at (-3.0,-0.1) {$G_1(\cdot)$};
      \node[below] at (-5.2,0.5) {{\scriptsize$i_1=1$}};
      \node[below] at (-5.0,1.2) {{\scriptsize$i_1=2$}};
      \node[below] at (-4.6,2.0) {{\scriptsize$i_1=n_1$}};

      \node[draw,circle,inner sep=0.3pt,fill] at (-0.5,2.0) {};

      \draw [-, thick] (0,0) -- (0,2);
      \draw [-, thick] (0,0) -- (2,0);
      \foreach \x in {1,3}
      \draw [-, thick] (\x*0.3,2+\x*0.5) -- (\x*0.3,1.5+\x*0.5);
      \foreach \x in {1,3}
      \draw [-, thick] (\x*0.3+1.7,\x*0.5) -- (\x*0.3+2,\x*0.5);
      \foreach \x in {0,1,3}
      \draw [-, thick] (\x*0.3,2+\x*0.5) -- (\x*0.3+2,2+\x*0.5);
      \foreach \x in {0,1,3}
      \draw [-, thick] (\x*0.3+2,2+\x*0.5) -- (\x*0.3+2,\x*0.5);
      \foreach \x in {0,1,2}
      \node[draw,circle,inner sep=0.3pt,fill] at (1.6+\x*0.15,2.73+\x*0.25) {};

      \node[below] at (1.0,-0.1) {$G_2(\cdot)$};

      \draw [-, thick] (6.7+ 0,0) -- (6.7+ 0,2);
      \draw [-, thick] (6.7+ 0,0) -- (6.7+ 2,0);
      \foreach \x in {1,3}
      \draw [-, thick] (6.7+ \x*0.3,2+\x*0.5) -- (6.7+ \x*0.3,1.5+\x*0.5);
      \foreach \x in {1,3}
      \draw [-, thick] (6.7+ \x*0.3+1.7,\x*0.5) -- (6.7+ \x*0.3+2,\x*0.5);
      \foreach \x in {0,1,3}
      \draw [-, thick] (6.7+ \x*0.3,2+\x*0.5) -- (6.7+ \x*0.3+2,2+\x*0.5);
      \foreach \x in {0,1,3}
      \draw [-, thick] (6.7+ \x*0.3+2,2+\x*0.5) -- (6.7+ \x*0.3+2,\x*0.5);
      \foreach \x in {0,1,2}
      \node[draw,circle,inner sep=0.3pt,fill] at (6.5+ 1.6+\x*0.15,2.73+\x*0.25) {};

      \node[below] at (7.8,-0.1) {$G_{\mu}(\cdot)$};
     \node[below] at (5.3,1.3) {{\scriptsize$i_\mu=1$}};
      \node[below] at (5.5,2.7) {{\scriptsize$i_\mu=2$}};
      \node[below] at (5.9,3.5) {{\scriptsize$i_\mu=n_\mu$}};
      
      \foreach \x in {0,1,2}
      \node[draw,circle,inner sep=0.3pt,fill] at (3.5+\x*0.2,2.0) {};

      \draw [-, thick] (11.0+ 0,0) -- (11.0+ 0,2);
      \draw [-, thick] (11.0+ 0,0) -- (11.0+ 2,0);
      \foreach \x in {1,3}
      \draw [-, thick] (11.0+ \x*0.3,2+\x*0.5) -- (11.0+ \x*0.3,1.5+\x*0.5);
      \foreach \x in {1,3}
      \draw [-, thick] (11.0+ \x*0.3+1.7,\x*0.5) -- (11.0+ \x*0.3+2,\x*0.5);
      \foreach \x in {0,1,3}
      \draw [-, thick] (11.0+ \x*0.3,2+\x*0.5) -- (11.0+ \x*0.3+2,2+\x*0.5);
      \foreach \x in {0,1,3}
      \draw [-, thick] (11.0+ \x*0.3+2,2+\x*0.5) -- (11.0+ \x*0.3+2,\x*0.5);
      \foreach \x in {0,1,2}
      \node[draw,circle,inner sep=0.3pt,fill] at (11.0+ 1.6+\x*0.15,2.73+\x*0.25) {};

      \node[below] at (12.0,-0.1) {$G_{d-1}(\cdot)$};
      \foreach \x in {0,1,2}
      \node[draw,circle,inner sep=0.3pt,fill] at (10.2+\x*0.2,2.0) {};
      
      \node[draw,circle,inner sep=0.3pt,fill] at (14.5,2.0) {};
      
      \draw [-, thick] (15.0+ 0,0) -- (15.0+ 0, 2);
      \draw [-, thick] (15.0+ 0,0) -- (15.0+ 0.5, 0);
      \foreach \x in {1,3}
      \draw [-, thick] (15.0+ \x*0.3, 2+\x*0.5) -- (15.0+ \x*0.3,  1.5+\x*0.5);
      \foreach \x in {1,3}
      \draw [-, thick] (15.0+ \x*0.3+0.2, \x*0.5) -- (15.0+ \x*0.3+0.5,  \x*0.5);
      \foreach \x in {0,1,3}
      \draw [-, thick] (15.0+ \x*0.3,2+\x*0.5) -- (15.0+ \x*0.3+0.5,  2+\x*0.5);
      \foreach \x in {0,1,3}
      \draw [-, thick] (15.0+ \x*0.3+0.5,2+\x*0.5) -- (15.0+ \x*0.3+0.5,  \x*0.5);
      \foreach \x in {0,1,2}
      \node[draw,circle,inner sep=0.3pt,fill] at (14.3+ 1.6+\x*0.15,2.73+\x*0.25) {};

      \node[below] at (15.3,-0.1) {$G_{d}(\cdot)$};
      
    \end{tikzpicture}
  \end{center}
  \caption{The TT representation of a tensor in $TT(r_1,\ldots,r_{d-1})$ with $G_\mu(i_\mu)\in\mathbb{R}^{r_{\mu-1}\times r_\mu}$.\label{mps}}
\end{figure}
\begin{definition}[TT tensor format]\label{TT tensor format}
  Let $r_0,\ldots,r_d\in\mathbb{N}$ and $r_0=r_d=1$. A tensor $A\in\mathbb{R}^\Index$ of the form or representation
  \begin{equation}\label{MPSform}
    A_{i_1,\ldots,i_d} = G_1(i_1)\cdots G_d(i_d),\qquad G_{\mu}(i_\mu)\in\mathbb{R}^{r_{\mu-1}\times r_{\mu}}
  \end{equation}
  for all $i\in\Index$ and $G_\mu:\Index_\mu\to\mathbb{R}^{r_{\mu-1}\times r_\mu}$ is said to 
  be of MPS (matrix product states) format or TT (tensor train) format or hierarchical format, cf. Figure \ref{mps}. 
  We define the set of tensors in TT format by
  \[TT(r_1,\ldots,r_{d-1}) := \{A\in\mathbb{R}^{\Index}\mid A \,\text{is of the form}\, (\ref{MPSform})\}.\]
  The parameters $r_\mu$ are called representation ranks and combined to the rank vector 
  {\boldmath$r$}$:=(r_1,\ldots,r_{d-1})$. For the matrix blocks $(G_\mu)_{\mu=1}^{d}$
  we use the short notation $G$. $G$ is called a representation system of $A$, and if we want to indicate
  that $A$ is represented by $G$ we write $A^G$.
\end{definition}
The minimal ranks $r_\mu$ for the representation of a tensor $A$ in TT format are the ranks 
of certain matricizations of $A$ \cite{Gr10,OsTy10}.

The number of parameters in the TT representation is 
\[ \sum_{\mu=1}^{d}r_{\mu-1}r_{\mu}n_\mu \sim {\cal O}(dr^2n),
\qquad r:=\max_{\mu\in D}\;r_\mu,
\quad n:=\max_{\mu\in D}\;n_\mu.\]
It could thus in principle be possible to reconstruct the tensor 
from a number of samples that is in ${\cal O}(\log N)$ for a tensor having $N=\prod_{i=1}^{d}n_i$ entries,
cf. Section \ref{sec:numrec}. 

\subsection{Statement of the Main Approximation Problem}

The full approximation problem can be stated as follows. For $S \subset \Index$ let
\[ \|X\|_{F} := \sqrt{\sum_{i\in \Index}X_i^2},\qquad (X|_S)_i := \begin{cases}
X_i & \mbox{ if } i \in S  \\
0 & \mbox{ otherwise } \end{cases}, \qquad \|X\|_{S}:=\|X|_S\|_F. \]

\begin{problem}[Main problem]\label{mainproblem}
  Given a tensor $M\in\mathbb{R}^{\Index}$ known only at points $P\subset\Index$, and given representation 
  ranks $r_1,\ldots,r_{d-1}$, find a representation (\ref{MPSform}) with representation system $G$ such that 
  $A=A^G$ fulfils
  \[A = \mathop{\rm argmin}_{\tilde{A}\in TT(r_1,\ldots,r_{d-1})} \|M-\tilde{A}\|_{P}.\]
\end{problem}

A related approach for tensor completion is presented in \cite{dSiHe13} where the authors use a steepest 
descent iteration on the tensor manifold. Our approach is an alternating least squares minimization and
an overrelaxation based on ideas from LMaFit for matrix completion \cite{WeYiZha10}. 
A short comparison is given in Section \ref{subsection:kl}.

\subsection{First Order Optimality Conditions and ALS}

For a representation system $(G_{\mu})_{\mu=1}^{d}$ such that $A=A^G$ one can write the 
main problem in the form
\[ G = \mathop{\rm argmin}_{\tilde{G}} \|M-A^{\tilde{G}}\|_{P}.\]
The direct first order optimality conditions for the matrix blocks $G_\mu$ are
\[G_\mu = \mathop{\rm argmin}_{\tilde{G}_\mu} \|M-A^{\tilde{G}}\|_{P},\quad
\tilde{G}_\nu := G_\nu \text{ for } \nu\ne\mu,\]
i.e. each matrix block $G_\mu$ is optimal when all other blocks are fixed. Starting
from some approximation $G$, the alternating least squares approach from \cite{HoRoSchn12} 
consists of an alternating best fit for each of the blocks $G_\mu$ in the order 
$\mu=1,\ldots,d$. It should be noted that the order can as well be chosen as
$\mu=d,\ldots,1$ or any other permutation. However, for practical purposes, the 
most straightforward choice seems to be either one of the aforementioned orderings, cf. Algorithm \ref{als}. 

\begin{algorithm}[htb]
  \begin{algorithmic}
    \REQUIRE Initial guess $A^G$
    \WHILE{stopping condition not fulfilled}
    \FOR{$\mu=1,\ldots,d$} 
    \STATE Determine $G_\mu := \mathop{\rm argmin}_{G_\mu}\|M-A^{G}\|_P $
    \ENDFOR
    \ENDWHILE
  \end{algorithmic}
  \caption{Alternating Least Squares (ALS) algorithm\label{als}}
\end{algorithm}

\begin{remark}[Slice-wise optimization]\label{swo}
  The minimizer $G_\mu$ in each step of Algorithm \ref{als} can be found slice-wise, 
  since each slice yields an independent least squares problem:
  \[
  G_{\mu}(j_{\mu}) := argmin_{G_\mu(j_{\mu})}\|M_{i_{\mu} = j_{\mu}}-A^{G}_{i_{\mu} = j_{\mu}} \|_{\{ p \in P \mid p_{\mu} = j_{\mu} \}}\]
\end{remark}

\subsection{Alternative Optimality Conditions and ADF}

An alternative formulation of our main problem is based on LMaFit ideas \cite{WeYiZha10}
and given by introducing an additional tensor $Z\in\mathbb{R}^\Index$ so that $G$ can be found via solving
\[ \mbox{minimize } f(G,Z) := \|Z-A^G\|_F 
\quad s.t.\quad Z|_{P} = M|_{P}, \quad A^G \in TT(r_1,\ldots,r_{d-1}).\]
The latter function $f$ yields first order optimality conditions
\[
Z|_{\Index \setminus P} = A^G|_{\Index \setminus P} \qquad \mbox{and}\quad 
G_\mu = \argmin_{\tilde{G}_\mu}\|Z-A^{\tilde{G}}\|_F,\quad
\tilde{G}_\nu := G_\nu \text{ for } \nu\ne\mu.\]
Solving this nonlinear system of equations simultaneously for $G_1,\ldots,G_d,Z$ is not trivial. 
In a hard or soft thresholding iteration, one would have to find a best approximation $A^G$ to a given tensor 
$Z$, and in the matrix case $d=2$ this is expensive but possible. For tensors in $d>2$ such a best approximation 
is not available. A common technique for finding a quasi-optimal approximation is an alternating optimization approach, cycling through
the unknowns $G_\mu$ (as above in ALS). But since our final goal is not the approximation of $Z$ but the minimization of $f$, 
it makes sense to directly solve the nonlinear system by an alternating fit.
We approach this nonlinear system by a nonlinear block Gauss-Seidel iteration where the blocks of
unknowns are $G_1,\ldots,G_d,Z$:
\begin{algorithmic}
  \REQUIRE Initial guess $A^G$
  \FOR{i=1,\ldots}
  \STATE For all $i\in\Index\setminus P$ set  $Z_i:=A^{G}_i$ and for all $i\in P$ set $Z_i:=M_i$
  \STATE For all $\mu\in D$ minimize $\|Z-A^{G}\|_F$ with respect to $G_\mu$ 
  \ENDFOR
\end{algorithmic}
Finally, we use (partial) successive overrelaxation in order to speed up the convergence.
We call the resulting algorithm `alternating directions fitting' (ADF), cf. Algorithm \ref{adf} 
(where the overrelaxation parameter still has to be specified).
\begin{algorithm}[htb]
  \begin{algorithmic}
    \REQUIRE Initial guess $A^G$, overrelaxation parameter $\alpha\ge 1$
    \WHILE{stopping condition not fulfilled}
    \STATE For all $i\in\Index\setminus P$ set  $Z_i:=A^{G}_i$ and for all $i\in P$ set $Z_i:=M_i$
    \FOR{$\mu=1,\ldots,d$} 
    \STATE Determine $G_\mu^+ := \argmin_{G_\mu}\|Z-A^{G}\|_F$ and set $G_\mu := G_\mu + \alpha(G_\mu^+-G_\mu)$
    \ENDFOR
    \ENDWHILE
  \end{algorithmic}
  \caption{Alternating Directions Fitting (ADF) algorithm\label{adf}}
\end{algorithm}

\subsection{Organization of the Article}

In Section \ref{sec:tensorcalculus}, we introduce the necessary tools for the analysis and algorithmic
treatment of the tensor approximation problem. 
Section \ref{sec:adfalgorithm} presents the ALS and ADF algorithm in detail and analyses the 
computational and storage complexity 
of one iterative step. Several practical issues like adaptive choice of the ranks, improved performance, 
and stopping criteria are developed. Finally we greatly simplify the determination of the overrelaxation 
parameter $\alpha$.
In the numerical examples in Section \ref{sec:numexp}, we apply the algorithms to three types of examples:
a) smooth function related tensors,
b) functionals of parametric PDE solutions, and
c) random low rank tensors with and without noise.
We conclude our findings in Section \ref{sec:conclusions}.

\section{Optimization in the TT-Format}\label{sec:tensorcalculus}
In this section we introduce the neccessary tools to work with matrix blocks in order to derive and formulate 
the core step of the ALS and ADF algorithm (Theorem \ref{corestep}).

\subsection{Matrix Blocks}

First we introduce matrix blocks, which are a useful tool both for tensor calculus and 
arithmetic in TT representation. 

\begin{definition}[Matrix block]
Let $k_1, k_2, n \in \mathbb{N}$. We define a matrix block $H \in (\mathbb{R}^{k_1 \times k_2})^n$ as a vector of matrices $H(1),\ldots,H(n) \in \mathbb{R}^{k_1 \times k_2}$. We call $k_1 \times k_2$ the dimension and $n$ the length of $H$.
\end{definition}
\begin{remark}
  a) In \cite{HoRoSchn12} a matrix block $H$ is called a component function $H(\cdot)$. 
  We use the name matrix block to point out that $H$ has the structure of an array of matrices.
  b) For fixed $k_1, k_2\in\mathbb{N}$ the set of matrix blocks $H \in (\mathbb{R}^{k_1 \times k_2})^n$ 
  forms an $\mathbb{R}$-vectorspace as well as a left-module over the non-abelian matrix ring 
  $\mathbb{R}^{k_1 \times k_1}$ and a right-module over $\mathbb{R}^{k_2 \times k_2}$.
\end{remark}

Matrix blocks can be combined via the Kronecker product to form higher dimensional tensors as they appear
in the definition of the TT representation $A^G$.
\begin{definition}[(Kronecker) product between matrix blocks]\label{kron}
We define the (Kronecker) product $\otimes$ for matrix blocks $H_1, H_2$ of dimensions $k_1 \times k_m, k_m \times k_2$ and lengths $n_1, n_2$ as
\[(H_1 \otimes H_2)((i,j)) := H_1(i) H_2(j) \]
where $(H_1 \otimes H_2)$ is a matrix block of dimension $k_1 \times k_2$ and length $n_1 n_2$.
\end{definition}
The definition is consistent with the conventional Kronecker product such that associativity is given. 

In order to simplify the notation we use the following convention:
\begin{itemize}
\item We treat the product of a matrix and a matrix block as if the matrix was a block of length $1$ and skip the 
  $\otimes$. It is referred to as pointwise multiplication. 
\item We write  $(H_1 \otimes \ldots \otimes H_n)(i_1 \ldots i_n)$  instead of 
  $(H_1 \otimes \ldots \otimes H_n)((i_1 \ldots i_n))$.
\item The empty Kronecker product is defined to be $I$ (identity matrix of suitable size).
\end{itemize}

\begin{remark}[Generating $A^G$]\label{gagbt}
Using the Kronecker product, one can express $A^G$ by
\[ A^G_{(i_1,\ldots,i_d)} = (G_1 \otimes \ldots \otimes G_d) (i_1,\ldots,i_d),\quad
A^G = G_1 \otimes \ldots \otimes G_d.\]
\end{remark}

In order to apply standard matrix tools, we have to switch between matrix blocks, matrices, and tensors. The
necessary foldings and unfoldings are introduced in the following.

\begin{definition}[Left and right unfolding, transpose]\label{hb}

Let $H \in (\mathbb{R}^{k_{1} \times k_{2}})^n$ be a matrix block. We define the left unfolding $\lhb{H}$ as
\begin{equation}\label{lefthandblock}
  \lhb{H} := \begin{bmatrix}
    H(1) \\
    H(2) \\
    \vdots \\
    H(n)
  \end{bmatrix} \in \mathbb{R}^{n k_{1} \times k_{2}},
  \qquad
  \mbox{\begin{minipage}[c]{4cm}
      \begin{tikzpicture}[label distance=-1mm,scale = 0.3]
        \draw [-, thick] (0,0) -- (0,2);
        \draw [-, thick] (0,0) -- (2,0);
        \foreach \x in {1,3}
	\draw [-, thick] (\x*0.3,2+\x*0.5) -- (\x*0.3,1.5+\x*0.5);
        \foreach \x in {1,3}
	\draw [-, thick] (\x*0.3+1.7,\x*0.5) -- (\x*0.3+2,\x*0.5);
        \foreach \x in {0,1,3}
	\draw [-, thick] (\x*0.3,2+\x*0.5) -- (\x*0.3+2,2+\x*0.5);
        \foreach \x in {0,1,3}
	\draw [-, thick] (\x*0.3+2,2+\x*0.5) -- (\x*0.3+2,\x*0.5);
        \foreach \x in {0,1,2}
	\node[draw,circle,inner sep=0.3pt,fill] at (1.6+\x*0.15,2.73+\x*0.25) {};
	
        \draw[->, thick] (4,1.75) -- node[below] {$\lhb{\cdot}$}  (6,1.75);
        
        \draw [-, thick] (7,3.75+1.75) -- (7,-0.5+1.75);
        \draw [-, thick] (9,3.75+1.75) -- (9,-0.5+1.75);
        \draw [-, thick] (9,3.75+1.75) -- (7,3.75+1.75);
        \draw [-, dashed] (9,1.75+1.75) -- (7,1.75+1.75);
        \draw [-, dashed] (9,-0.25+1.75) -- (7,-0.25+1.75);
        \draw [-, dashed] (9,-1.75+1.75) -- (7,-1.75+1.75);
        \draw [-, thick] (9,-3.75+1.75) -- (7,-3.75+1.75);
        \draw [-, thick] (9,-3.75+1.75) -- (9,-1.5+1.75);
        \draw [-, thick] (7,-3.75+1.75) -- (7,-1.5+1.75);
        \foreach \x in {1,2,3}
	\node[draw,circle,inner sep=0.3pt,fill] at (8,1.5/4*\x) {};
      \end{tikzpicture}
    \end{minipage}
  }
\end{equation}
and the right unfolding $\rhb{H}$ as
\begin{equation}\label{righthandblock}
  \rhb{H} := \begin{bmatrix}
    H(1) & H(2) & \ldots & H(n) 
  \end{bmatrix} \in \mathbb{R}^{k_{1} \times n k_{2}}.
  \qquad
  \mbox{\begin{minipage}[c]{4.5cm}
      \begin{tikzpicture}[label distance=-1mm,scale = 0.3]
        \draw [-, thick] (0,0) -- (0,2);
        \draw [-, thick] (0,0) -- (2,0);
        \foreach \x in {1,3}
	\draw [-, thick] (\x*0.3,2+\x*0.5) -- (\x*0.3,1.5+\x*0.5);
        \foreach \x in {1,3}
	\draw [-, thick] (\x*0.3+1.7,\x*0.5) -- (\x*0.3+2,\x*0.5);
        \foreach \x in {0,1,3}
	\draw [-, thick] (\x*0.3,2+\x*0.5) -- (\x*0.3+2,2+\x*0.5);
        \foreach \x in {0,1,3}
	\draw [-, thick] (\x*0.3+2,2+\x*0.5) -- (\x*0.3+2,\x*0.5);
        \foreach \x in {0,1,2}
	\node[draw,circle,inner sep=0.3pt,fill] at (1.6+\x*0.15,2.73+\x*0.25) {};
	
        \draw[->, thick] (4,1.75) -- node[below] {$\rhb{\cdot}$}  (6,1.75);
        
        \draw [-, thick] (-3.75-1.75+12.5,7-6.25) -- (+0.5-1.75+12.5,7-6.25);
        \draw [-, thick] (-3.75-1.75+12.5,9-6.25) -- (+0.5-1.75+12.5,9-6.25);
        \draw [-, thick] (-3.75-1.75+12.5,9-6.25) -- (-3.75-1.75+12.5,7-6.25);
        \draw [-, dashed] (-1.75-1.75+12.5,9-6.25) -- (-1.75-1.75+12.5,7-6.25);
        \draw [-, dashed] (+0.25-1.75+12.5,9-6.25) -- (+0.25-1.75+12.5,7-6.25);
        \draw [-, dashed] (+1.75-1.75+12.5,9-6.25) -- (+1.75-1.75+12.5,7-6.25);
        \draw [-, thick] (+3.75-1.75+12.5,9-6.25) -- (+3.75-1.75+12.5,7-6.25);
        \draw [-, thick] (+3.75-1.75+12.5,9-6.25) -- (+1.5-1.75+12.5,9-6.25);
        \draw [-, thick] (+3.75-1.75+12.5,7-6.25) -- (+1.5-1.75+12.5,7-6.25);
        \foreach \x in {1,2,3}
	\node[draw,circle,inner sep=0.3pt,fill] at (1.5/4*\x+11,8-6.25) {};
      \end{tikzpicture}
    \end{minipage}
  }
\end{equation}
The transpose $H^T$ of a matrix block is a matrix block defined by $H^T(i) := H(i)^T$.
\end{definition}

In \cite{RoUsch12} the left and right unfoldings $\lhb{H}$ and $\rhb{H}$ are denoted by $H^L$ and $H^R$. 
We adjust the notation to our requirements and in order to illustrate that they are mappings.

\begin{remark}[Conjugacy of block operations]\label{cobo}
The left and right unfolding are conjugate operations by means of
\[ \lhb{H}^T = \rhb{H^T} \]
\end{remark}
\begin{definition}[Left and right $s$-unfolding of a representation]\label{larpoar}
For a representation $G$ as in Definition \ref{TT tensor format}, we denote the left $s$-unfolding by 
\[ G^{<s} := \lhb{G_1 \otimes \ldots \otimes G_{s-1}} \in \mathbb{R}^{n^{<s} \times r_{s-1}}, \quad n^{<s} = \prod_{\mu < s} n_{\mu} \]
and likewise the right $s$-unfolding by 
\[ G^{>s} := \rhb{G_{s+1} \otimes \ldots \otimes G_d} \in \mathbb{R}^{r_s \times n^{>s}}, \quad n^{>s} = \prod_{\mu > s} n_{\mu}. \]
We shortly call these just unfoldings and skip the index $s$.
\end{definition}

\begin{figure}
  \begin{center}
    \begin{tikzpicture}[label distance=-1mm,scale = 0.35]
      
      \draw [-, thick] (0,0) -- (0,2);
      \draw [-, thick] (0,0) -- (2,0);
      \foreach \x in {1,3}
      \draw [-, thick] (\x*0.3,2+\x*0.5) -- (\x*0.3,1.5+\x*0.5);
      \foreach \x in {1,3}
      \draw [-, thick] (\x*0.3+1.7,\x*0.5) -- (\x*0.3+2,\x*0.5);
      \foreach \x in {0,1,3}
      \draw [-, thick] (\x*0.3,2+\x*0.5) -- (\x*0.3+2,2+\x*0.5);
      \foreach \x in {0,1,3}
      \draw [-, thick] (\x*0.3+2,2+\x*0.5) -- (\x*0.3+2,\x*0.5);
      \foreach \x in {0,1,2}
      \node[draw,circle,inner sep=0.3pt,fill] at (1.6+\x*0.15,2.73+\x*0.25) {};
      
      \draw [-, thick] (7-11,3.75+1.75) -- (7-11,-0.5+1.75);
      \draw [-, thick] (9-11,3.75+1.75) -- (9-11,-0.5+1.75);
      \draw [-, thick] (9-11,3.75+1.75) -- (7-11,3.75+1.75);
      \draw [-, dashed] (9-11,1.75+1.75) -- (7-11,1.75+1.75);
      \draw [-, dashed] (9-11,-0.25+1.75) -- (7-11,-0.25+1.75);
      \draw [-, dashed] (9-11,-1.75+1.75) -- (7-11,-1.75+1.75);
      \draw [-, thick] (9-11,-3.75+1.75) -- (7-11,-3.75+1.75);
      \draw [-, thick] (9-11,-3.75+1.75) -- (9-11,-1.5+1.75);
      \draw [-, thick] (7-11,-3.75+1.75) -- (7-11,-1.5+1.75);
      \foreach \x  in {1,2,3,5,6,7,12,13,14}
      \draw [-, dashed] (9-11,3.75+1.75-\x*0.5) -- (7-11,3.75+1.75-\x*0.5);
      \foreach \x in {1,2,3}
      \node[draw,circle,inner sep=0.3pt,fill] at (8-11,1.5/4*\x) {};
      
      \draw [-, thick] (-3.75-1.75+10.5,7-6.25) -- (+0.5-1.75+10.5,7-6.25);
      \draw [-, thick] (-3.75-1.75+10.5,9-6.25) -- (+0.5-1.75+10.5,9-6.25);
      \draw [-, thick] (-3.75-1.75+10.5,9-6.25) -- (-3.75-1.75+10.5,7-6.25);
      \draw [-, dashed] (-1.75-1.75+10.5,9-6.25) -- (-1.75-1.75+10.5,7-6.25);
      \draw [-, dashed] (+0.25-1.75+10.5,9-6.25) -- (+0.25-1.75+10.5,7-6.25);
      \draw [-, dashed] (+1.75-1.75+10.5,9-6.25) -- (+1.75-1.75+10.5,7-6.25);
      \draw [-, thick] (+3.75-1.75+10.5,9-6.25) -- (+3.75-1.75+10.5,7-6.25);
      \draw [-, thick] (+3.75-1.75+10.5,9-6.25) -- (+1.5-1.75+10.5,9-6.25);
      \draw [-, thick] (+3.75-1.75+10.5,7-6.25) -- (+1.5-1.75+10.5,7-6.25);
      \foreach \x in {1,2,3}
      \node[draw,circle,inner sep=0.3pt,fill] at (1.5/4*\x+9,8-6.25) {};
      \foreach \x  in {1,2,3,5,6,7,12,13,14}
      \draw [-, dashed] (-3.75-1.75+10.5+\x*0.5,7-6.25) -- (-3.75-1.75+10.5+\x*0.5,9-6.25);
      
      \node at (-6.25,1.75) {$=$};
      
      \begin{scope}[shift={(-1,-1)}]
        \draw [-, thick] (-15.5+0.3,5.5+0.5) -- (-8+0.3,5.5+0.5);
        \draw [-, thick] (-15.5+0.3,5.5+0.5) -- (-15.5+0.3,5.5);
        \draw [-, thick] (-8+0.3,5.5+0.5) -- (-8+0.3,-2+0.5);
        \draw [-, thick] (-8,-2+0.5) -- (-8+0.3,-2+0.5);
        \begin{scope}[shift={(0.6,1)}]
          \draw [-, thick] (-15.5+0.3,5.5+0.5) -- (-8+0.3,5.5+0.5);
          \draw [-, thick] (-15.5+0.3,5.5+0.5) -- (-15.5+0.3,5.5);
          \draw [-, thick] (-8+0.3,5.5+0.5) -- (-8+0.3,-2+0.5);
          \draw [-, thick] (-8,-2+0.5) -- (-8+0.3,-2+0.5);
        \end{scope}
        \foreach \x in {0,1,2}
	\node[draw,circle,inner sep=0.3pt,fill] at (3.75-15.5+\x*0.15+0.5,6.25+\x*0.25) {};
        
        \draw [-, thick] (-8,3.75+1.75) -- (-8,-0.5+1.75);
        \draw [-, thick] (-3.75-1.75+10.5-20.5,3.75+1.75) -- (-3.75-1.75+10.5-20.5,-0.5+1.75);
        \draw [-, thick] (-8,3.75+1.75) -- (-10,3.75+1.75);
        \draw [-, thick] (-8,-3.75+1.75) -- (-10,-3.75+1.75);
        \draw [-, thick] (-8,-3.75+1.75) -- (-8,-1.5+1.75);
        \draw [-, thick] (-3.75-1.75+10.5-20.5,-3.75+1.75) -- (-3.75-1.75+10.5-20.5,-1.5+1.75);
        \foreach \x  in {1,2,3,4,5,6,7,8,11,12,13,14}
	\draw [-, dashed] (-10,3.75+1.75-\x*0.5) -- (-8,3.75+1.75-\x*0.5);
	\foreach \x  in {1,2,3,4,5,6,7,8,11,12,13,14}
	\draw [-, dashed] (-11.5,3.75+1.75-\x*0.5) -- (-3.75-1.75+10.5-20.5,3.75+1.75-\x*0.5);
        \foreach \x in {1,2,3}
	\node[draw,circle,inner sep=0.3pt,fill] at (-9,1.5/4*\x) {};
	\foreach \x in {1,2,3}
	\node[draw,circle,inner sep=0.3pt,fill] at (-13.5,1.5/4*\x) {};
	
        \draw [-, thick] (-3.75-1.75+10.5-20.5,-3.75+1.75) -- (+0.5-1.75+10.5-20.5,-3.75+1.75);
        \draw [-, thick] (-3.75-1.75+10.5-20.5,3.75+1.75) -- (+0.5-1.75+10.5-20.5,3.75+1.75);
        \draw [-, thick] (-3.75-1.75+10.5-20.5,0) -- (-3.75-1.75+10.5-20.5,-3.75+1.75);
        \draw [-, thick] (+3.75-1.75+10.5-20.5,0) -- (+3.75-1.75+10.5-20.5,-3.75+1.75);
        \draw [-, thick] (+3.75-1.75+10.5-20.5,-3.75+1.75) -- (+1.5-1.75+10.5-20.5,-3.75+1.75);
        \draw [-, thick] (+3.75-1.75+10.5-20.5,+3.75+1.75) -- (+1.5-1.75+10.5-20.5,+3.75+1.75);
        \foreach \x in {1,2,3}
	\node[draw,circle,inner sep=0.3pt,fill] at (1.5/4*\x+9-20.5,-1) {};
	\foreach \x in {1,2,3}
	\node[draw,circle,inner sep=0.3pt,fill] at (1.5/4*\x+9-20.5,3.5) {};
        \foreach \x  in {1,2,3,4,5,6,7,8,11,12,13,14}
	\draw [-, dashed] (-3.75-1.75+10.5+\x*0.5-20.5,0) -- (-3.75-1.75+10.5+\x*0.5-20.5,-2);
        \foreach \x  in {1,2,3,4,5,6,7,8,11,12,13,14}
	\draw [-, dashed] (-3.75-1.75+10.5+\x*0.5-20.5,1.5) -- (-3.75-1.75+10.5+\x*0.5-20.5,3.75+1.75);
      \end{scope}
    \end{tikzpicture}
  \end{center}
  \caption{The block matricization of $A^G$ is the product $A^G_{(s)} = G^{<s} \ G_s \ G^{>s}$ of the left
    unfolding times matrix block times right unfolding.\label{pmoag}}
\end{figure}
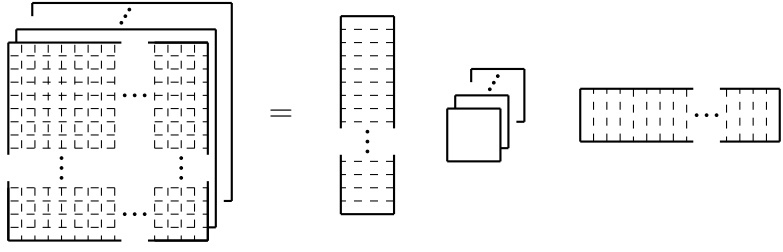
\begin{definition}[Block matricization]\label{pm}
Let $A \in \mathbb{R}^{\Index}$ be a d-dimensional tensor. A block matricization with respect to $s \in \{1,\ldots,d\}$, $A_{(s)}$, is defined as the matrix block of dimension $(n_1\ldots n_{s-1}) \times (n_{s+1}\ldots n_d)$ and length $n_s$, given by 
\[ (A_{(s)}(i_s))_{(i_1,\ldots,i_{s-1}),(i_{s+1},\ldots,i_d)} := A_{i_1,\ldots,i_d}, \quad \forall i_s\in\Index_s.\]
\end{definition}
In case that $A$ is a tensor in TT format with representation $A=A^G$, the block matricization is
simply (cf. Figure \ref{pmoag})
\[ A^G_{(s)} = G^{<s} \ G_s \ G^{>s}\]

\subsection{Scalar Product and Orthogonality}

The standard scalar product can be transfered to matrix blocks as follows.

\begin{definition}[(Scalar) product of matrix blocks]\label{sp}
Let $G$ and $H$ be matrix blocks of dimensions $k_1 \times k_m$,$k_m \times k_2$ and same length. Then we define 
their (scalar) product as
\[ \langle G,H\rangle := \sum_{\substack{i}} G(i) H(i) = \rhb{G} \lhb{H} \in \mathbb{R}^{k_1 \times k_2} .\]
For a matrix $J \in \mathbb{R}^{k_{m} \times k_{m}}$ we define
\[ \langle G,J,H\rangle := \langle G J,H\rangle = \langle G,J H\rangle. \]
Note that $\langle \cdot,\cdot \rangle$ is only a product with scalar output regarding its module properties. 
\end{definition}

\begin{definition}[$\mathbb{R}$-scalar product and matrix block norm]
Let $V := (\mathbb{R}^{k_1 \times k_2})^n$ be the $\mathbb{R}$-vector space of matrix blocks of dimension $k_1 \times k_2$ and length $n$. Then $\langle\cdot,\cdot\rangle$ defines a scalar product $\langle \cdot,\cdot \rangle_\mathbb{R}$ on $V$ via
\[ \langle G,H \rangle_\mathbb{R} := {\rm trace} \langle G, H^T \rangle = {\rm trace} \langle G^T, H \rangle, 
\qquad G,H\in V.\]
The corresponding norm $||\cdot||$ on $V$ is defined as $\|G\| := \sqrt{\langle G, G \rangle_{\mathbb{R}}}$.
\end{definition}

\begin{remark}[Properties of the matrix block norm and scalar product]
For a matrix block $G$, tensor $A$ and index $s\in D$, it holds 
\[ \|G\| = \sqrt{\sum_{i} ||G(i)||^2_F}, \qquad \|A\|_F = \|A_{(s)}\|. \]
The $\mathbb{R}$ scalar product hence coincides with the standard scalar product between the according vectorizations of the matrix blocks.
\end{remark}

We introduce the concept of orthogonality (cf. \cite{HoRoSchn12}) for matrix blocks, by which we can 
simplify the minimization problem.

\begin{definition}[Orthogonality of matrix blocks]\label{ortho}
For a matrix block $H$, we call $H$ \\
left orthogonal if the columns of $\lhb{H}$ are orthogonal (this being $\langle H^T,H\rangle = I$), and
right orthogonal if the rows of $\rhb{H}$ are orthogonal (this being $\langle H,H^T\rangle = I$). \\
Let $Q$ be a matrix block of same dimensions as $H$. We then define the (non-unique) operation $orth^{\ell}$ such that for $Q = orth^{\ell}(H)$, the pair $(\lhb{Q},R)$ is a QR-decomposition of $\lhb{H}$. Then $Q$ is left orthogonal and $Q R = H$. \\
Likewise $orth^r$ is such that for $Q = orth^r(H)$, the pair $(L,\rhb{Q})$ is an LQ-decomposition of $\rhb{H}$. 
Then $Q$ is right orthogonal and $L Q = H$.
\end{definition}
In Corollary \ref{ookp}, we demonstrate how orthogonality, the scalar product and the Kronecker product 
are used to show the feasibility (Theorem \ref{lucs}) of the ADF core step (Theorem \ref{corestep}).
\begin{lemma}[Scalar products of Kronecker products]\label{spoakp}
Let $G_1, G_2$ and $H_1, H_2$ be matrix blocks of appropriate dimensions and lengths. Then
\[ \langle (G_1 \otimes G_2)^T, H_1 \otimes H_2 \rangle = \langle G_2^T, \langle G_1^T, H_1 \rangle, H_2 \rangle, \]
respectively
\[ \langle G_1 \otimes G_2, (H_1 \otimes H_2)^T \rangle = \langle G_1, \langle G_2, H_2^T \rangle, H_1^T \rangle. \]
\end{lemma}
\begin{proof}
Due to symmetry we consider only the first case. By definition and reordering of summation, we obtain
\begin{align*}
  & \langle (G_1 \otimes G_2)^T, H_1 \otimes H_2 \rangle \quad  = \sum_{i} ((G_1 \otimes G_2)(i))^T (H_1 \otimes H_2)(i) \\
  = & \sum_{i_1,i_2} G_2(i_2)^T G_1(i_1)^T (H_1(i_1) H_2(i_2)) \quad
  = \sum_{i_2} G_2(i_2)^T \sum_{i_1} (G_1(i_1)^T H_1(i_1)) H_2(i_2) \\
  = & \sum_{i_2} G_2(i_2)^T \langle G_1^T, H_1 \rangle H_2(i_2) \quad
  = \langle G_2^T, \langle G_1^T, H_1 \rangle, H_2 \rangle.
\end{align*}
\end{proof}

\begin{corollary}[Orthogonality of Kronecker products]\label{ookp}
If $G_1=H_1$ are left orthogonal in Lemma \ref{spoakp}, then
\[\langle (G_1 \otimes G_2)^T, H_1 \otimes H_2 \rangle = \langle G_2^T,H_2 \rangle.\]
If $G_2=H_2$ are right orthogonal in Lemma \ref{spoakp}, then
\[\langle G_1 \otimes G_2, (H_1 \otimes H_2)^T \rangle = \langle G_1,H_1^T \rangle.\]
\end{corollary}
\begin{remark}[Non-uniqueness of representations]\label{noninj}
In the TT-format, the representations are highly non-unique \cite{RoUsch12}. This degree of freedom can be an 
advantage: one can always assume that all matrix blocks $G_i$ are left orthogonal for $i<h$ and right
orthogonal for $i>h$. Then $G$ is called orthogonalized with respect to $h$, or in short $h$-orthogonal. 
This concept is also described in \cite{HoRoSchn12}, where $G_h$ is called core of $G$. 
It follows that $\left\|A^G\right\|_F = \left\|G_h\right\|$.
\end{remark}


\section{The ALS and ADF Algorithm}\label{sec:adfalgorithm}

We first approach Problem \ref{mainproblem} by the ALS Algorithm \ref{als}, for which we introduce the rank 
increasing strategy in detail in Algorithm \ref{alsri}.
We then derive the optimality conditions of this problem with respect to a single block, which is the basic step of 
the ADF Algorithm \ref{adf}.
We adapt the stopping criteria, previously given for the rank increasing ALS algorithm, and 
provide a useful heuristic for choosing the overrelaxation parameter $\alpha$ 
(Remark \ref{alpha} and Algorithm \ref{dir}). Finally, we greatly simplify the choice of $\alpha$.


\subsection{Rank Increasing Strategy and Alternating Least Squares}

In this section we assume that a target rank $r_{final}$ is given and that we are interested in a tensor 
completion scheme with equal ranks $r_1 = \cdots = r_{d-1} = r_{final}$ in the TT format. 
A successful strategy for finding good initial values for the optimization is to start with 
minimal ranks $r_1 = \ldots = r_{d-1} = 1$. Each time the algorithm fails to progress sufficiently
(cf. Remark \ref{bcls}), the ranks $r_\mu$ of $G$ are increased until the  
final target rank $r_{final}$ is reached. 

\begin{remark}[Initial values]\label{iv}\label{athr}
  We start our approximation scheme with equal ranks $r_1 = \ldots = r_{d-1} = 1$ and matrix blocks 
  \[ (G_s(i))_{1,1} := \frac{1}{\sqrt{n}}, \quad \forall s, i. \]
  $G$ is thereby uniform and each block is orthogonal. 
  The adaption of the representation $G$ to ranks $r + 1$ is done in a 
  straightforward way. The two matrix blocks $G_1, G_d$ are replaced by
  \[ G_1(i) \leftarrow \begin{bmatrix} G_1(i) & 1/\sqrt{n} \end{bmatrix}, \ G_d(i) \leftarrow \begin{bmatrix} G_d(i) \\ 1/\sqrt{n} \end{bmatrix} \qquad \forall i, \]
  while the other matrix blocks are replaced by
  \[ G_s(i) \leftarrow \begin{bmatrix} G_s(i) & 0 \\ 0 & 1/\sqrt{n} \end{bmatrix} \qquad \forall i. \] 
  This results in an initial guess which is the sum of the previous (lower) rank approximation plus
  a rank one term as above.
\end{remark}

\begin{remark}[Stopping criteria]\label{bcls}
  Our rank increasing scheme needs a robust stopping criterion for the least squares fixed rank 
  optimization. Here, we use the heuristic that whenever the improvements of one sweep are too small,
  we stop the fixed rank optimization and increase the rank parameter, where 'sweep' refers to one alternating cycle through all directions.
  Let $\mean{\gamma}_5$ denote the arithmetic mean of the last $5$ residual reduction 
  factors ($Res(G):=\|A^G-M\|_P$ after a sweep):
  \[ \gamma_i := \frac{Res(G^i)}{Res(G^{i-1})},\qquad  i={\tt iter}-4 ,\ldots, {\tt iter},\qquad
  \mean{\gamma}_5 := \frac{\gamma_{{\tt iter}-4} +\cdots +\gamma_{\tt iter}}{5}.\]
  Then we stop the fixed rank optimization if 
  \[|1-\mean{\gamma}_5| < \varepsilon_{stop}\] 
  where reasonable choices for $\varepsilon_{stop}$ vary between $10^{-2}$ and $10^{-5}$. 
 
\end{remark}

The final algorithm with our choice of starting values is given 
in Algorithm \ref{alsri}.
The orthogonalization of $G$ with respect to $s$ in the inner loop is not necessary 
but improves the stability and can be performed without significant increase in 
computational complexity. 

\begin{algorithm}[!htb]
  \caption{Rank increasing ALS algorithm\label{alsri}}
  \begin{algorithmic}
    \STATE Initialize the representation $G$ for $r = 1$ (Remark \ref{iv}); 
    \FOR{$r=1 \ldots r_{final}$}
    \FOR{${\tt iter} = 1, \ldots, {\tt iter}_{max}$}
	\FOR{$s = 1,\ldots,d$}
		\STATE orthogonalize $G$ with respect to $s$;
		\STATE update $G_s \leftarrow \argmin_{G_s} \|A^G - M\|_P$;

	 \ENDFOR
	 \IF{stopping criteria apply}     
	 \STATE stop the {\tt iter} loop;
	 \COMMENT {Remark \ref{bcls}}
	 \ENDIF
    \ENDFOR
	\STATE adapt representation to $r+1$; 
	\COMMENT{Remark \ref{athr}}
    \ENDFOR
  \end{algorithmic}
\end{algorithm}

\begin{lemma}[Computational complexity]\label{re:comp}
  The computational complexity for one sweep of the ALS algorithm for rank $r$ is in 
  \[\mathcal{O}(r^4 d \# P)\]
  Assuming that for each rank $r=1,\ldots,r_{final}$ we require ${\tt iter}_{max}$ many sweeps of
  ALS, we obtain a total complexity of
  \[\mathcal{O}({\tt iter}_{max} r^5 d \# P)\]
\end{lemma}
\begin{proof}
  The estimate for the total complexity obviously follows from the first one. 
  For one sweep we have to determine each of the blocks $G_s$ once by setting up and solving 
  a linear least squares problem. Naturally the least squares problem decouples into $n_s$
  independent linear least squares problems of size
  $\#{\{ p \ | \ p \in P, p_{s} = i_{s} \}} \times r^2$.
  Solving these for all $i_s=1,\ldots,n_s$ is possible in 
  ${\cal O}(\#P r^4)$, and summing this up for all directions $s=1,\ldots,d$ gives a 
  complexity of ${\cal O}(d \#P r^4)$. 
  In addition to the pure solve, we have to setup the least squares matrix, and we
  orthogonalize $G$ with respect to $s$. 

  The orthogonalization step is independent of $P$ and of negligible complexity ${\cal O}(dnr^3)$ \cite{Gr10,Os11}.
  Setting up the least squares matrix requires $\#P$ times the (partial) evaluation of the tensor $A^G$, which is of
  complexity ${\cal O}(dr^2)$ per entry, leading to a negligible complexity of ${\cal O}(r^2 d \# P)$.
\end{proof}

For the convergence of the ALS iteration, we state the result from \cite[Theorem 2.10]{RoUsch12}: under suitable
full rank assumptions on the Hessian in the local minimizer, the ALS iteration converges locally 
at least linearly to the local minimizer.
 
\subsection{The ADF Core Step}

The core step we outline below describes how the update of $G$ in Algorithm \ref{adf} in the unaccelerated case is performed.

\begin{theorem}[Core step of the ADF algorithm]\label{corestep}
  Let $s\in\{1,\ldots,d\}$. Without loss of generality, we assume that 
  $G$ is orthogonalized with respect to $s$ (cf. Remark \ref{noninj}).

  Then the minimizer $G_s$ in Algorithm \ref{adf}, for all $j\in \Index_s$, is given by
  \begin{align*}
  G_s(j) & = (G^{<s})^T \ Z_{(s)}(j) \ (G^{>s})^T \\ & = \sum_{i\in \Index, i_s=j}
  Z_{i} (G_1(i_1)\ldots G_{s-1}(i_{s-1}))^T (G_{s+1}(i_{s+1})\ldots G_d(i_d))^T.  \end{align*}

\end{theorem}

\begin{proof} By assumption, $G_1, \ldots, G_{s-1}$ are left orthogonal and $G_{s+1},\ldots, G_d$ right orthogonal. 
Therefore $G^{<s}$ has orthonormal columns and $G^{>s}$ orthonormal rows. Then
  \[
    G_s = \argmin_{G_s} \|Z-A^G\|_F 
    = \argmin_{G_s}   \|Z_{(s)}-A^G_{(s)}\|
    = \argmin_{G_s}  \|Z_{(s)}-G^{<s} \ G_s \ G^{>s}\|
   \]

   and, due to orthogonality, it follows that, for all $j\in \Index_s$,
  \[
    G_s(j) = \argmin_{G_s(j)} \|Z_{(s)}(j)-G^{<s} \ G_s(j) \ G^{>s}\| 
    = \argmin_{G_s(j)} \|(G^{<s})^T \ Z_{(s)}(j) \ (G^{>s})^T -G_s(j)\|.
  \]

\end{proof}

The core step above is formulated without any overrelaxation. The overrelaxation parameter $\alpha$ can 
however be included directly into the core step by modifying $Z$ as follows.

\begin{lemma}\label{z}
  Let $A^G=G_1\otimes\cdots\otimes G_d\in\mathbb{R}^\Index$ be given, $\alpha\in\mathbb{R}$, $Z\in\mathbb{R}^\Index$
  and 
  \[G_s^+ := \argmin_{\tilde{G}_s}\|Z_{(s)} 
  - G^{<s} \ \tilde{G}_s \ G^{>s} \|. \]
  Then $G^{\alpha}_s := \alpha G_{s}^+ + (1-\alpha) G_s$ satisfies
  \begin{equation} 
    G^{\alpha}_s = \argmin_{\tilde{G}_s}\|Z^\alpha_{(s)} - G^{<s} \ \tilde{G}_s \ G^{>s}\|
    \label{eq:overr}
  \end{equation}
  for $Z^\alpha := \alpha Z + (1-\alpha)A^G$. 
\end{lemma}
\begin{proof}
We assume, by contradiction, that there exists $\hat{G}_s\neq G^{\alpha}_s$ satisfying $\hat{G}_s = \alpha \hat{G}^+_s + (1-\alpha)G_s$ and
 \begin{equation*}
   \|Z^\alpha_{(s)} - G^{<s} \ \hat{G}_s \ G^{>s} \|< \|Z^\alpha_{(s)} - G^{<s} G^{\alpha}_s \ G^{>s} \|.
 \end{equation*}
 Inserting  $Z^\alpha$, $\hat{G}_s$ and $G^{\alpha}_s$ leads to
 \begin{align*}
   & \|Z^\alpha_{(s)} - G^{<s} \ \hat{G}_s \ G^{>s} \|=\| \alpha Z_{(s)}- \alpha G^{<s} \ \hat{G}^+_s \ G^{>s} + (1-\alpha) (A^G_{(s)}-G^{<s} \ G_s \ G^{>s})\|\\
   & < \|Z^\alpha_{(s)} - G^{<s} \ G^{\alpha}_s \ G^{>s} \|=\| \alpha Z_{(s)}- \alpha G^{<s} \ G_{s}^+ \ G^{>s} + (1-\alpha) (A^G_{(s)}-G^{<s} \ G_s \ G^{>s})\|
 \end{align*}
 which is equivalent to
 \begin{equation*}
   \alpha \|Z_{(s)}-G_{<s} \  \hat{G}^+_s \ G^{>s}\| < \alpha \|Z_{(s)}-G_{<s} \  G_{s}^+ \ G^{>s}\|.
 \end{equation*}
 This is a contradiction to the minimality of $G_{s}^+$. This proves that $G^{\alpha}_s$ is the minimizer of the minimization problem (\ref{eq:overr}).
\end{proof}

\begin{remark}[Denoting current and old representations within sweeps]\label{dcaor} The intermediate tensor $Z^{\alpha}$ is not updated along with the representation, but
in chosen increments, namely after each sweep. During each sweep, we denote with $G^-$ the old representation used for the last update of $Z^{\alpha}$ and with $G$
 the current representation. Therefore $Z^{\alpha}$ is always based on the old representation.
\end{remark}

\begin{theorem}[Practical ADF core step]\label{lucs}
  Under the assumptions of Theorem \ref{corestep}, the update block $G_s$ with 
  overrelaxation parameter $\alpha$ is given, for all $j\in\Index_s$, by
  \begin{align} 
    G_s(j) 
    = & \label{summand1}
    \underbrace{(G^{<s})^T \ (G^-)^{<s}}_{(LS^1_s)} 
    \, G^{-}_s(j) \,
    \underbrace{(G^-)^{>s} \ (G^{>s})^T}_{(LS^2_s)} \\
    & + \label{summand2}
    \sum_{\substack{i \in P, i_s = j}} \alpha (M_i - A^{G^{-}}_i) \,
    \underbrace{(G_1(i_1)\ldots G_{s-1}(i_{s-1}))^T}_{(LM^1_s)_i} \,
    \underbrace{(G_{s+1}(i_{s+1})\ldots G_d(i_d))^T}_{(LM^2_s)_i}
  \end{align}
  (The short notations are used for Lemma \ref{sc}.) 
\end{theorem}
\begin{proof} According to Theorem \ref{corestep} and Lemma \ref{z}, we have
  \begin{equation} 
    G_s(j) = (G^{<s})^T \ Z^\alpha_{(s)}(j) \ (G^{>s})^T.
    \label{gsz}
   \end{equation}
  $Z = A^{G^{-}}|_{\Index \setminus P} + M|_P$ (cf. Algorithm \ref{adf}) and $Z^{\alpha} = \alpha Z + (1-\alpha) A^{G^{-}} $ (cf. Lemma \ref{z}) yield
  \[ Z^\alpha = \underbrace{A^{G^{-}}}_{\hookrightarrow \text{First summand}} 
  + \underbrace{\alpha (M|_P - A^{G^{-}}|_P)}_{\hookrightarrow \text{Second summand}}\]
  which we insert into (\ref{gsz}).\\
  \textbf{First summand:} 
  Recall that $A^{G^{-}}_{(s)}$ can be expanded (cf. Figure \ref{pmoag}). From the definition of
  $(G^{<s})$ and $(G^{>s})$ (cf. Definition \ref{larpoar}), we derive that
  \begin{equation} 
    (G^{<s})^T \ A^{G^{-}}_{(s)}(j) \ (G^{>s})^T = (G^{<s})^T \ (G^-)^{<s}
    \ G^{-}_s(j) \ (G^-)^{>s} \ (G^{>s})^T
    \label{qap}
  \end{equation}
 
  \textbf{Second summand:} 
  As $(M|_P - A^{G^-}|_P)_i = 0$ for all $i \notin P$, we can reduce the summation from $\Index$ to $P$ and obtain the formula stated in the theorem. 
\end{proof}

\subsection{Computational Complexity of ADF}

The statements presented in this subsection are based on the sweep with order $1 \rightarrow d$, but can be 
transfered to permutations.

\begin{lemma}[Successive computing]\label{sc}
  The occuring terms in the core step (Theorem \ref{lucs}) during the sweep ($s = 1 \rightarrow d$) can be reduced 
  to simpler successive computations. Note that in step $s$, the right matrix blocks $G_{s+1},\ldots,G_d$ are 
  unchanged and equal to those of the old representation $G^-$. We then have that
  \begin{align} (LS^1_{s}) = \langle G^T_{s-1}, (LS^1_{s-1}), G^-_{s-1} \rangle, \label{sc1}
  \end{align}
  where $(LS^1_1) = 1$, while $(LS^2_s) = I$ (the identity matrix) due to the orthogonality conditions. Likewise
  \begin{align}
    (LM^1_s)_i & = G_{s-1}(i_{s-1})^T \ (LM^1_{s-1})_i , \label{sc3}\\
    (LM^2_s)_i & = (LM^2_s)_i \, G^-_{s+1}(i_{s+1})^T   \label{sc4}
  \end{align}
  where $(LM^1_1) = 1$. Hence, while $(LS^1)$ and $(LM^1)$ are updated within the sequence, $(LM^2)$ is 
  calculated before. Furthermore, $(LM^1_s)$ and $(LM^2_s)$ can be used to update $A^{G^-}|_P$.
\end{lemma}

\begin{lemma}[Computational complexity]\label{pi}
  Let $r := max\{r_1,\ldots,r_{d-1}\}$ and $n:= max\{n_1,\ldots,n_d\}$.
  The complexity for one full sweep of updating $Z,G_1,\ldots,G_D$ in the ADF iteration is
  \[{\cal O}(r^3dn + r^2d\#P).\]
\end{lemma}
\begin{proof}
We analyze the operations in Lemma \ref{sc} and Theorem \ref{lucs} for a step $s$ within a sweep $s=1\to d$:
\begin{enumerate}
\item $(\ref{sc1})$: $2n$ times an $(r \times r)$ times $(r \times r)$ matrix multiplication: 
  ${\cal O}(nr^3)$.
\item $(\ref{sc3})$ \& $(\ref{sc4})$: 2$\#P$ times an $(1 \times r)$ times $(r \times r)$ matrix multiplication:
  ${\cal O}(\#Pr^2)$.
\item $(\ref{summand1})$: $2n$ times an $(r \times r)$ times $(r \times r)$ matrix multiplication:
  ${\cal O}(nr^3)$.
\item $p$ times evaluation of $A^{G^-}$, by using the values $(LM^1_s)$,$(LM^2_s)$: 
  ${\cal O}(\#Pr^2)$.
\item $(\ref{summand2})$: $\#P$ times an $(r \times 1)$ times $(1 \times r)$ matrix multiplication:
  ${\cal O}(\#Pr^2)$.  
\item switching orthogonality of $G$: one QR decomposition of an $nr \times r$ matrix and 
  $n$ times an $(r \times r)$ times $(r \times r)$ matrix multiplication:
  ${\cal O}(nr^3)$.
\end{enumerate}
Each of these steps is performed ${\cal O}(d)$ times. 

This leaves us with the computational complexity of one left-hand sweep of
$\mathcal{O}(r^3dn + r^2d\#P)$. 

\end{proof}

\begin{remark}[Complexity of ALS and ADF]
  The computational complexity of one ADF sweep is in ${\cal O}(r^3dn+r^2d\#P)$, whereas
  an ALS sweep is in ${\cal O}(r^4d\#P)$ (cf. Lemma \ref{re:comp}), i.e., 
  asymptotically an ADF step is by a factor $r^2$ faster than an ALS step.
  In the numerical examples section we compare the speed and the necessary number of iterations
  for several examples.  
\end{remark}

\subsection{Preliminary choice of the SOR Parameter $\alpha$ and Stopping Criterion}

By an optimized determination of the acceleration parameter $\alpha$, one can speed up the convergence of the ADF algorithm considerably. Therefore, after each sweep of the ADF Algorithm \ref{adf}, we allow a relatively expensive search for a suitable $\alpha$ by testing increased ($\alpha^{up}$) and reduced ($\alpha^{down}$) values of $\alpha$ until the residual decays (or we break). The corresponding representations are denoted by $G^{up},G^{down}$, and the
direction (up, down or back) is denoted by ${\tt dir}$. 
The residual error is denoted as above by $Res(G):=\|A^G-M\|_P$.

\begin{remark}[Determination of the overrelaxation $\alpha$]\label{alpha}\label{iap}
  To handle the acceleration parameter $\alpha$, we introduce a second parameter $\delta$, an increment parameter.
  Each sweep is run for two different accelerations
  ($\alpha^{up},\alpha^{down}$):
  \[ \alpha^{up} := \alpha + \delta, \quad \alpha^{down} := max\{1, \alpha - \delta / 5\}. \]
  This choice ensures that the overrelaxation parameter is at least $\alpha\ge 1$.
  Depending on the residuals of the results, one of the three directions  
  is chosen as specified in Algorithm \ref{dir}. 
  It determines the new $\alpha$, $\delta$ as well as $G$. 
  
  In order to estimate and understand the magnitude of $\alpha$, 
  one can view the summand (\ref{summand2}) as a spot-check evaluation of the same term but
  for $P = \Index$, which would represent a full, maximal sampling set. 
  Therefore, it has to be multiplied by $\frac{\#\Ind}{\# P}$. 
  For the initial acceleration parameters needed for the ADF algorithm, we obtain 
  \[ \alpha := \frac{\#\Ind}{\# P}, \quad \delta := \frac{\alpha}{4}. \] 
\end{remark}

\begin{algorithm}[htb]
  \caption{Choice of the SOR parameter $\alpha$\label{dir}}
  \begin{center}
     \fbox{{\bf Notation:} {
        $\searrow$ $\delta$ means $\delta := \frac{1}{2}\delta$,
        \hspace{0.5cm}
        $\nearrow$ $\delta$ means $\delta := \min(\alpha^{back}/10,1.2 \delta)$ 
    } }
  \end{center}
  \begin{algorithmic}
    \STATE The values $G^{up}$ and $G^{down}$ for overrelaxations $\alpha^{up}$ and $\alpha^{down}$ are already computed
    \IF{$Res(G^{up}) > Res(G)$ and $Res(G^{down}) > Res(G)$}
    \STATE Set $\alpha:=\frac{1}{2}(1+\alpha), \searrow\delta$ and ${\tt dir}:=back$, then
    recompute $G^{up}$ and $G^{down}$ \\
    Restart Algorithm \ref{dir} (break if this happens more than $10$ times in a row);
    \ELSIF{$(Res(G^{up}) < Res(G^{down}))$} 
    \STATE If ${\tt dir}=up$ then $\nearrow$ $\delta$, otherwise $\searrow$ $\delta$; 
    \STATE $\alpha := \alpha^{up}$;  $G := G^{up}$;  ${\tt dir} := up$;
    \ELSIF{$(Res(G^{down}) \le Res(G^{up}))$}
    \STATE If ${\tt dir}=down$ then $\nearrow$ $\delta$, otherwise $\searrow$ $\delta$; 
    \STATE $\alpha := \alpha^{down}$; $G := G^{down}$;  ${\tt dir} := down$;
    \ENDIF
  \end{algorithmic}
\end{algorithm}

Finally, we need an adaptive reliable stopping criterion in conjunction with the rank-increasing strategy discussed previously.

\begin{remark}[Stopping criteria]\label{bc}
  We denote again by $\mean{\gamma}_5$ the arithmetic mean of the last $5$ residual reduction factors
  \[\gamma_i := \frac{Res(G^i)}{Res(G^{i-1})},\qquad  i={\tt iter}-4 \ldots {\tt iter}.\]
  Our first stopping criterion is simply like for ALS
  \[|1-\mean{\gamma}_5| < \varepsilon_{stop}\] 
  with $\varepsilon$ between $10^{-2}$ and $10^{-5}$.
 
  However, this is only tested if the direction is ${\tt dir}=down$ or the last residual reduction fulfils:
  $|1-\frac{\gamma_i}{\gamma_{i-1}}| < 10^{-7}$. 
  Note that we cannot compare the specific $\varepsilon_{stop}$ of ADF with the one of ALS, as ADF is faster in time but
  with smaller residual reduction per iteration.
  Our second stopping criterion is: 
  Stop if the last $10$ directions were ${\tt dir}=back$, meaning
  there is no residual reduction even if the SOR parameter $\alpha$ approaches 1.
\end{remark}

A detailed analysis by numerical experiments on the optimality of $\alpha$ from the above heuristic is given in 
the supplementary material. We can summarize that even an expensive line search to determine 
the optimal $\alpha$ for each sweep gives almost the same results as the simple heuristic. 
This motivates the simplified determination of $\alpha$ in the next subsection.

\subsection{Automated Overrelaxation in Microsteps}

The idea for the automated overrelaxation is not to choose one $\alpha$ for the whole
sweep $s=1\to d$ but rather a different $\alpha=\alpha(s)$ for each (micro-) step $s$. 
As it will turn out this enables us to determine the optimal $\alpha(s)$ and interprete
the iteration as an approximate ALS iteration.  

\begin{definition}[Residual tensor and matrix block projection]\label{ResT}
  We define the residual tensor $S$ and the matrix block projection $P_s$ via
  \begin{align*}
    S^G_M & := (M - A^G)|_P, \quad (A_{(s)})|_{P_s} := (A|_P)_{(s)},
  \end{align*}
  such that for any $s \in D$: $(S^G_M)_{(s)} = (M_{(s)} - G^{<s} \ G_s \ G^{>s})|_{P_s} $. 
  When the context is clear, we skip the indices $M$ or $G$.
\end{definition}

We recall that the tensor $Z^\alpha$ is given by
\[ Z^\alpha = A^{G^{-}} + \alpha (M|_P - A^{G^{-}}|_P) = A^{G^{-}} + \alpha S^{G^-}. \]
If we assume that $G=G^-$ is $s$-orthogonal and we determine the update only in direction $s$
(instead of the whole sweep $1\to d$), then the update used in ADF simplifies to

\begin{align*} 
  G^{\alpha}_s(j) & =  (G^{<s})^T \ Z^\alpha_{(s)}(j) \ (G^{>s})^T. \\
  & =  
  \underbrace{(G^{<s})^T \ (G^-)^{<s}}_{= I} 
  \, G^{-}_s(j) \,
  \underbrace{(G^-)^{>s} \ (G^{>s})^T}_{= I} \\
  & \quad + \alpha
  \underbrace{\sum_{\substack{i \in P, i_s = j}}  (M_i - A^{G^{-}}_i) \,
    (G_1(i_1)\ldots G_{s-1}(i_{s-1}))^T \,
    (G_{s+1}(i_{s+1})\ldots G_d(i_d))^T}_{=: N(j)} \\
  & =  \, G^{-}_s(j) \, + \alpha N(j).
\end{align*}
For the whole matrix block this is $N = {G^{<s}}^T \ S^{G^-}_{(s)} \ {G^{>s}}^T$. 

\begin{lemma}[Optimal acceleration]\label{AOG_alpha}
  The optimal overrelaxation parameter 
  \[
  \alpha^* := \alpha^*(s) := \argmin_{\alpha} \|G^{<s} \ G^{\alpha}_s \ G^{>s} - M_{(s)}\|_{P_s}
  \]
  for the update of block $s$ is given by
  \[
  \alpha^* = \| N \|^2_F /   \| G^{<s} \ N \ G^{>s} \|^2_{P_s}.
  \]  
\end{lemma}
\begin{proof}
  The optimal $\alpha^*$ from the quadratic minimization is
  \[\alpha^*=
  \langle G^{<s} \ N  \ G^{>s},\, (M_{(s)} - G^{<s} \ G^{-}_s  \ G^{>s})|_{P_s}\rangle / \| G^{<s} \ N \ G^{>s}\|^2_{P_s}.
  \]
Finally, the trace properties can be used to simplify the nominator:
  \begin{align*} 
    & \,\,\, \langle G^{<s} \ N \ G^{>s}, (M_{(s)} - G^{<s} \ G^{-}_s  \ G^{>s})_{P_s} \rangle \\
    & =  \sum_{j=1}^n \ trace((G^{<s} \ N(j) \ G^{>s})^T \ (M_{(s)}(j) - G^{<s} \ G^{-}_s(j)  \ G^{>s})_{P_s}) \\
    &=  \sum_{j=1}^n \ trace(N(j)^T \ {G^{<s}}^T \ (M_{(s)}(j) - G^{<s} \ G^{-}_s(j)  \ G^{>s})_{P_s} \ {G^{>s}}^T) \\
    &=  \langle N, {G^{<s}}^T \ (M_{(s)} - G^{<s} \ G^{-}_s  \ G^{>s})_{P_s} \ {G^{>s}}^T) \rangle \\
    &=  \langle N, {G^{<s}}^T \ S^{G^-}_{(s)} \ {G^{>s}}^T \rangle  = \langle N, N \rangle 
  \end{align*} 
\end{proof}

Note that the change in the residual tensor has already been calculated for the determination of $\alpha^*$:
$S^G_{(s)} = S^{G^-}_{(s)} - \alpha \ (G^{<s} \ N \ G^{>s})|_{P_s}$. 
Furthermore $\|S^{G^-}\|^2 - \|S_{G}\|^2 = \alpha \ ||N||^2 $. 
We summarize the final ADF in Algorithm \ref{AOG_alg}.

\begin{algorithm}[!htb]
  \caption{Rank increasing ADF algorithm\label{AOG_alg}}
  \begin{algorithmic}
    \STATE Initialize the representation $G$ for $r = 1$ (Remark \ref{iv});
    $S := S^G_M$ (Definition \ref{ResT});
    \FOR{$r=1,\ldots,r_{final}$}
      \FOR{${\tt iter} = 1, \ldots, {\tt iter}_{max}$}
	\FOR{$s=1,\ldots,d$}
	  \STATE $s$-orthogonalize $G$ and calculate  
	  \[ N := {G^{<s}}^T \ S_{(s)} \ {G^{>s}}^T, \quad Z_N := ({G^{<s}} \ N \ {G^{>s}})_{P_s} \] 
	  \STATE set $\alpha := \frac{ \| N \|^2_F}{  \| Z_N \|^2_{P_s}}$; \COMMENT{Lemma \ref{AOG_alpha}, Remark \ref{nua} resp.}
	  \STATE update
	  \[ G_s := G_s + \alpha N, \quad S_{(s)} := S_{(s)} - \alpha Z_N \]
	\ENDFOR
	\IF{breaking criteria apply}     
	\STATE stop the {\tt iter} loop;
	\COMMENT {Remark \ref{bcls}}
	\ENDIF
      \ENDFOR
      \STATE adapt representation to $r+1$; 
      \COMMENT{Remark \ref{athr}}
    \ENDFOR
  \end{algorithmic}
\end{algorithm}

\begin{remark}[Overrelaxation in Microsteps]\label{nua}
  The overrelaxation parameter $\alpha$ does not need to be uniform for the 
  whole block $G_s$. We can proceed with each part $G_s(j), j = 1,\ldots,n_s$ 
  seperately due to their independency. $N$ and $Z_N$ remain the same and the 
  optimal $\alpha^*_j$ for slice $j$ is
  \[
  \alpha_j^* = \frac{ \| N(j) \|^2_F}{  \| G^{<s} \ N(j) \ G^{>s} \|^2_{P^j_s}}, 
  \quad \mbox{for } (A_{(s)}(j))|_{P^j_s} := (A|_P)_{(s)}(j), \ j = 1,\ldots,n_s.
  \] 
  Hence, we update $G_s(j) = {G^-_s}(j) + \alpha_j N(j)$. This is what we use in practice as 
  it typically gives a lower residual for the same computational complexity.
\end{remark}

Finally, we can interprete the ADF iteration with overrelaxation in microsteps as an 
approximate ALS iteration: 
The block $N$ as defined above is the gradient of the residual function in mode $s$. 
That is, for
\[
  R_s:  \ \R^{r_{s-1} \times r_{s} \times n_s} \rightarrow \R,\quad
   G_s \mapsto \frac{1}{2} \| M - A^G \|_P^2,
\]
we have $N = \nabla R_s$. Therefore the ADF (micro-) step is an alternating best approximation 
of the blocks $G_s$, $s=1\to d$, but only in the direction $N$ of steepest descent (after 
$s$-orthogonalization). In the
numerical examples we observe that indeed ADF requires a few more iterative steps, but since
the complexity is by a factor $r^2$ lower, this is advantageous.

\section{Numerical Experiments}\label{sec:numexp}

\subsection{Data Aquisition and Measurements}

\textit{Sampling:} In order to obtain a sufficient slice density, cf. Definition \ref{def:fibre_density}, we generate the set
$P$ in a quasi-random way as follows: For each direction $\mu = 1,\ldots,d$ and each index 
$i_{\mu}\in\Index_\mu$ we pick $C_{SD} r^2$ indices $i_1,\ldots,i_{\mu-1},i_{\mu+1},\ldots,i_d$ 
at random (uniformly). This gives in total $\# P = d n C_{SD} r^2$ samples (excluding some exceptions), 
where $C_{SD}$ is the slice density from Definition \ref{def:fibre_density}. 
  As a control set $C$, we use a set of the same cardinality as $P$ that is generated in the same way. 

\textit{Stopping parameter:} We give neither a limit to time nor to the number of iterations and use only 
the previously mentioned stopping criteria where the $\varepsilon_{stop}$ for ADF is always $1/3$ the one 
for ALS.
The different choices for $\varepsilon_{stop}$ are to compensate for the differing per-iteration computational 
complexity of each algorithm (and lead to a fair comparison). 

\textit{Order of optimization:} Furthermore we use a slightly different order of optimization as previously discussed. Instead of the sweep we gave before ($s = 1,\ldots,d$), we
alternate between two sweeps ($s = 1,\ldots,h, \quad s = d,\ldots,h, \quad h = \lfloor d/2 \rfloor$) to enhance symmetry. A full alternating sweep
($s = 1,\ldots,d, \quad s = d,\ldots,1$) can also be considered. However, we found that this sweep is slightly less effective.

\textit{Notation:} For the results of the tests we denote the ratio of known points $\rho=\#P/n^d$, the relative residual $res_P=\|A-X\|_P / \|A \|_P$, the error on the
control set $res_C=\|A-X\|_C / \|A\|_C$ and the $time$ in seconds. In order to save space, we 
sometimes label the y-axis above plots.

\subsection{Approximation of a Full Rank Tensor with Decaying Singular Values}\label{sec:invnormtensor}

As a first example, we consider a tensor $A\in\R^\Ind$ given by the entries
\begin{equation}
    \label{int} A_{(i_1,\ldots,i_d)} := \left(\sum_{\mu=1}^d i_\mu^2 \right)^{-1/2}.
\end{equation}

\begin{remark}[Approximation by exponential sums]\label{exp_sum}
  A good low-rank approximation of the aforementioned tensor $A$ (\ref{int}) can be obtained easily 
  from the following observation.  
  For any desired precision $\varepsilon\in(0,1)$ and $R>1$ there is a
  $k\in{\cal O}(\log(\varepsilon)\log(R))$ such that
  \begin{equation}\label{1sqrtr} 
    \forall r \in [1,R]:\quad
      \left\vert\frac{1}{\sqrt{r}} - \sum_{i=1}^k \omega_i^* e^{-\alpha_i^* r}\right\vert < \varepsilon,
  \end{equation}
  for specific values of $\omega^*, \alpha^*$ that depend on the desired accuracy $\varepsilon$ and
  upper bound $R$. The particular values can be obtained, cf. \cite{Ha12}, from the following webpage:
  \begin{center}
    \verb#http://www.mis.mpg.de/scicomp/EXP_SUM#
  \end{center}
  To transfer this observation to the multidimensional case, we insert $r = \|x\|^2, x \in \{1,\ldots,n\}^d$  
  and transform 
  \[ e^{-\|x\|^2} = e^{-\sum_{s=1}^d x_s^2} = \prod_{s=1}^d e^{-x_s^2}. \]
  Since, in this case, $r \in [d,dn^2]$, we rescale $\omega = \frac{1}{\sqrt{d}} \omega^*$ 
  and $\alpha = \frac{1}{d} \alpha^*$ as well as require that $R \ge n^2$. We finally obtain
  \[ A_{(i_1,\ldots,i_d)} 
  = \left(\sum_{s=1}^d i_s^2 \right)^{-1/2} \approx \sum_{\ell=1}^k \omega_{\ell} \prod_{s=1}^d e^{-\alpha_{\ell} i_s^2}. \]
  This yields a TT format respresentation $A=A^G$ with square diagonal matrices  
  \[ 
  (G_1(m))_{1,i} = \omega_i e^{-\alpha_i m^2},\quad
  (G_s(m))_{i,i} = e^{-\alpha_i m^2},\quad
  (G_d(m))_{i,1} = e^{-\alpha_i m^2},\]
  for $i = 1,\ldots,k$, $s = 2,\ldots,d-1$, $m = 1,\ldots,n$,
  of rank {\boldmath$r$} $= (k,\ldots,k)$ with a maximal pointwise error of $\frac{\varepsilon}{\sqrt{d}}$. 
  A rank $k$ approximation obtained in this way is not optimal in the sense that the same accuracy can be 
  reached with a smaller rank. In order to find the near best approximation, we make use of the 
  hierarchcial SVD (cf. \cite{Gr10}): In the first step we compute a highly accurate large rank
  tensor $\hat{A}\in TT(${\boldmath$\hat{r}$}$)$, in the second step we determine the quasi-optimal 
  approximation $A\in TT(${\boldmath$r$}$)$, 
  $\|A-\hat{A}\|\le \sqrt{d-1}\inf_{B\in TT(\mbox{\boldmath\scriptsize$r$})} \|B-\hat{A}\|$, cf. \cite{Os11,Gr10},
  by truncation of $\hat{A}$ to rank {\boldmath$r$} via the hierarchical SVD.
\end{remark}

We give convergence plots for varying target rank and slice density and also carry out four detailed, different tests, each one focusing on a different parameter: 
$d$ (dimension), $r$ (final rank), $n$ (size) and $C_{SD}$ (slice density).
In these tests we also compare the ADF with the ALS algorithm with stopping parameter 
$\varepsilon_{stop}=5\times 10^{-5}$ (ADF), and 
$\varepsilon_{stop}=15\times 10^{-5}$ (ALS).
 
Each combination of parameters is tested $20$ times for different random $P$ and $C$, where the same random instance of these parameters $P$ and $C$ is used in both ALS and ADF tests. Furthermore $\mean{res_C}$ and $\mean{res_P}$ denote the geometric mean of the respective results and 
$\mean{time}$ the arithmetic mean of times. The values in brackets give the geometric variance, respectively in case of the time the arithmetic variance.
A plot of the convergence of $\mean{res_P},\mean{res_C}$ for fixed $d=7$ , $n=12$ and varying target rank as well as slice density is given in Figure \ref{epsplot}.
\begin{figure}[htb]
  \begin{center}
    \begin{minipage}{\linewidth}
      \newlength\figureheight
      \newlength\figurewidth
    \setlength\figureheight{5cm}
    \setlength\figurewidth{0.9\linewidth}
%
%
\begin{tikzpicture}

\begin{axis}[%
width=0.410625\figurewidth,
height=\figureheight,
at={(0.540296\figurewidth,0\figureheight)},
scale only axis,
xmin=0,
xmax=0.0032,
xlabel={sampling ratio},
ymode=log,
ymin=1e-08,
ymax=0.01,
yminorticks=true,
title={relative residual on P}
]
\addplot [color=blue,dashed,mark size=3.5pt,mark=+,mark options={solid},forget plot]
  table[row sep=crcr]{%
2.81e-05	0.00258\\
6.33e-05	0.000173\\
0.000113	1.68e-05\\
0.000176	2.21e-06\\
0.000253	4.06e-07\\
0.000345	1.05e-07\\
0.00045	3.08e-08\\
};
\addplot [color=red,dashed,mark size=3.5pt,mark=o,mark options={solid},forget plot]
  table[row sep=crcr]{%
9.38e-05	0.00406\\
0.000211	0.000377\\
0.000375	3.64e-05\\
0.000586	4.29e-06\\
0.000844	7.86e-07\\
0.00115	1.51e-07\\
0.0015	3.43e-08\\
};
\addplot [color=black,dashed,mark size=3.5pt,mark=asterisk,mark options={solid},forget plot]
  table[row sep=crcr]{%
0.000188	0.00512\\
0.000422	0.00057\\
0.00075	5.2e-05\\
0.00117	7.54e-06\\
0.00169	1.04e-06\\
0.0023	2.13e-07\\
0.003	4.76e-08\\
};
\addplot [color=blue,solid,mark size=3.5pt,mark=+,mark options={solid},forget plot]
  table[row sep=crcr]{%
2.81e-05	0.00258\\
6.33e-05	0.000171\\
0.000113	1.44e-05\\
0.000176	2.19e-06\\
0.000253	5.57e-07\\
0.000345	1.89e-07\\
0.00045	7.58e-08\\
};
\addplot [color=red,solid,mark size=3.5pt,mark=o,mark options={solid},forget plot]
  table[row sep=crcr]{%
9.38e-05	0.00406\\
0.000211	0.000376\\
0.000375	3.6e-05\\
0.000586	4.06e-06\\
0.000844	8.66e-07\\
0.00115	2.11e-07\\
0.0015	8.11e-08\\
};
\addplot [color=black,solid,mark size=3.5pt,mark=asterisk,mark options={solid},forget plot]
  table[row sep=crcr]{%
0.000188	0.00512\\
0.000422	0.00057\\
0.00075	5.19e-05\\
0.00117	7.15e-06\\
0.00169	1.09e-06\\
0.0023	2.66e-07\\
0.003	8.17e-08\\
};
\end{axis}

\begin{axis}[%
width=0.410625\figurewidth,
height=\figureheight,
at={(0\figurewidth,0\figureheight)},
scale only axis,
xmin=0,
xmax=0.0032,
xlabel={sampling ratio},
ymode=log,
ymin=2e-05,
ymax=0.01,
yminorticks=true,
title={relative residual on C},
legend style={legend cell align=left,align=left,draw=white!15!black,font=\footnotesize}
]
\addplot [color=blue,dashed,mark size=3.5pt,mark=+,mark options={solid}]
  table[row sep=crcr]{%
2.81e-05	0.00827\\
6.33e-05	0.00282\\
0.000113	0.00139\\
0.000176	0.00121\\
0.000253	0.00058\\
0.000345	0.000653\\
0.00045	0.000239\\
};
\addlegendentry{$\text{C}_{\text{SD}}\text{=3 (ALS)}$};

\addplot [color=red,dashed,mark size=3.5pt,mark=o,mark options={solid}]
  table[row sep=crcr]{%
9.38e-05	0.00814\\
0.000211	0.00191\\
0.000375	0.000583\\
0.000586	0.000459\\
0.000844	0.000305\\
0.00115	0.000158\\
0.0015	0.000177\\
};
\addlegendentry{$\text{C}_{\text{SD}}\text{=10 (ALS)}$};

\addplot [color=black,dashed,mark size=3.5pt,mark=asterisk,mark options={solid}]
  table[row sep=crcr]{%
0.000188	0.00925\\
0.000422	0.002\\
0.00075	0.000623\\
0.00117	0.000266\\
0.00169	0.000121\\
0.0023	9.56e-05\\
0.003	8.4e-05\\
};
\addlegendentry{$\text{C}_{\text{SD}}\text{=20 (ALS)}$};

\addplot [color=blue,solid,mark size=3.5pt,mark=+,mark options={solid}]
  table[row sep=crcr]{%
2.81e-05	0.00825\\
6.33e-05	0.00281\\
0.000113	0.00118\\
0.000176	0.00115\\
0.000253	0.000568\\
0.000345	0.000651\\
0.00045	0.000236\\
};
\addlegendentry{$\text{C}_{\text{SD}}\text{=3 (ADF)}$};

\addplot [color=red,solid,mark size=3.5pt,mark=o,mark options={solid}]
  table[row sep=crcr]{%
9.38e-05	0.00814\\
0.000211	0.0019\\
0.000375	0.000571\\
0.000586	0.000499\\
0.000844	0.000291\\
0.00115	0.00015\\
0.0015	0.000215\\
};
\addlegendentry{$\text{C}_{\text{SD}}\text{=10 (ADF)}$};

\addplot [color=black,solid,mark size=3.5pt,mark=asterisk,mark options={solid}]
  table[row sep=crcr]{%
0.000188	0.00925\\
0.000422	0.002\\
0.00075	0.000621\\
0.00117	0.000266\\
0.00169	0.000137\\
0.0023	9.6e-05\\
0.003	8.44e-05\\
};
\addlegendentry{$\text{C}_{\text{SD}}\text{=20 (ADF)}$};

\end{axis}
\end{tikzpicture}%
    \end{minipage}
  \end{center}
\caption{\label{epsplot} ($d=7$, $r=2,\ldots,8$, $n=12$, $C_{SD} = 3,10,20$) Plotted are the residuals $\mean{res_P}$ (right) as well as the
   control residuals $\mean{res_C}$ (left) as function of the sampling ratio 
   $\rho = d n r^2 C_{SD}/n^d$ for varying target ranks $r = 2,\ldots,8$ indicated by the respective symbols. 
   Each curve corresponds to one choice of the slice density $C_{SD}$, for either ALS (dashed) or ADF (continous).}
\end{figure}
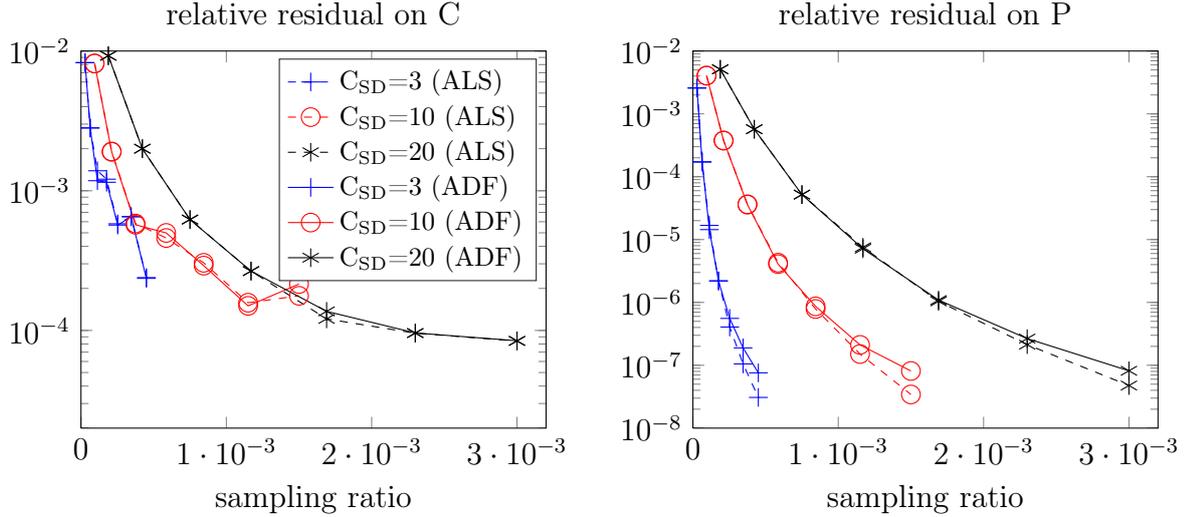
We observe convergence for all choices of parameters. In the Tables \ref{dtest}, \ref{rtest}, \ref{ntest} and \ref{ctest} we list the detailed 
results of the four mentioned comparisons (left: ALS, right: ADF).

First, we consider the variation of the dimension $d \in \{5,6,7,8,13,21,34,55\}$ in Table \ref{dtest}.
For all dimensions $d=5,\ldots,55$ the approximation seems to be uniformly good and the variance with 
respect to the randomness in the sampling points seems to be quite low. 
\begin{table}[htb]
\centering
\resizebox{\columnwidth}{!}{%
\begin{tabular}{|c|c|c|c||c|c|c|} 
\hline 
\multicolumn{7}{|c|}{$d$ varying, $r = 3$, $n = 8$, $C_{SD} =10$}\\ 
\hline 
& \multicolumn{3}{|c||}{ALS} & \multicolumn{3}{|c|}{ADF} \\ 
\hline  
$d$ & $\mean{res_C}$ & $\mean{res_P}$ & $\mean{time}$ & $\mean{res_C}$ & $\mean{res_P}$ & $\mean{time}$ \\ 
\hline \hline 
  5 & 2.9e-03(1.6) & 9.6e-04(1.2) & 0.1(0.0) & 2.9e-03(1.6) & 9.6e-04(1.2) & 0.1(0.0) \\ 
\hline \hline 
  6 & 2.2e-03(1.8) & 4.9e-04(1.3) & 0.2(0.0) & 2.2e-03(1.8) & 4.9e-04(1.3) & 0.1(0.0) \\ 
\hline \hline 
  7 & 1.2e-03(1.8) & 3.1e-04(1.2) & 0.3(0.1) & 1.2e-03(1.8) & 3.1e-04(1.2) & 0.2(0.1) \\ 
\hline \hline 
  8 & 1.2e-03(2.0) & 1.7e-04(1.2) & 0.4(0.1) & 1.2e-03(2.0) & 1.7e-04(1.2) & 0.3(0.1) \\ 
\hline \hline 
 13 & 1.8e-04(1.8) & 3.5e-05(1.1) & 1.7(0.4) & 1.8e-04(1.8) & 3.5e-05(1.1) & 1.0(0.2) \\ 
\hline \hline 
 21 & 3.7e-05(1.6) & 7.5e-06(1.2) & 7.4(3.2) & 3.7e-05(1.6) & 7.4e-06(1.2) & 4.0(1.9) \\ 
\hline \hline 
 34 & 7.5e-06(1.6) & 1.7e-06(1.1) & 24.9(8.3) & 7.4e-06(1.6) & 1.7e-06(1.1) & 14.0(4.3) \\ 
\hline \hline 
 55 & 1.4e-06(1.5) & 3.7e-07(1.1) & 71.2(27.7) & 1.4e-06(1.5) & 3.7e-07(1.1) & 42.5(17.6) \\ 
\hline 
\end{tabular} 
}
\vspace{0.1cm}
\caption{\label{dtest} Convergence and timing with respect to the dimension $d$ for otherwise fixed parameters.}
\end{table}

\begin{remark}(Comparison with HTOpt) 
  As a comparison of our results with the HTOpt algorithm from \cite{dSiHe13,dSiHe14} we perform the first 
  three test of Table \ref{dtest}, i.e. dimension $d\in\{5,6,7\}$, $r=3$, $n=8$, $\csd=10$. We have 
  used the default values provided by the program but set the maximal number of iterations to $1000$.
  The following table shows the approximation quality on the control set $C$ and given point set
  $P$, the accuracy of the near best exponential sum approximation ($res_{\rm exp}$) from Remark 
  \ref{exp_sum}, and the number of iterative steps:
  \begin{center}
    \begin{tabular}{|l|cc|c|r|}
      \hline
      $d$ & $\mean{res_C}$ & $\mean{res_P}$ & $res_{\rm exp}$ & steps\\
      \hline
      5 & 6.9e-03 & 2.4e-03 & 2.3e-03 & 519 \\
      6 & 2.7e-02 & 2.6e-03 & 1.5e-03 & 750\\
      7 & 5.2e-02 & 4.8e-03 & 1.0e-03   & 579\\
      \hline
    \end{tabular}
  \end{center}
  We can clearly see that the number of iterations in HTOpt used to find the approximation is rather stable. 
  We have used the dense linear algebra version provided in MATLAB, but a sparse version is also
  available. It seems that for smaller dimension the optimization on the manifold yields an approximation 
  close to the best one, whereas for larger dimension $d$ the quality diminishes.
\end{remark}

In the second experiment we vary the target ranks $r\in\{2,\ldots,8\}$ and report the results in 
Table \ref{rtest}. The approximation quality on the reference set $P$ is as we expected (exponentially
decaying to zero), but on the control set $C$ the fixed slice density $\csd$ limits the 
accuracy that we can achieve by the random sampling.
\begin{table}[htb]
\centering
\resizebox{\columnwidth}{!}{%
\begin{tabular}{|c|c|c|c||c|c|c|} 
\hline 
\multicolumn{7}{|c|}{$d = 7$, $r$ varying, $n = 12$, $C_{SD} =10$}\\ 
\hline 
& \multicolumn{3}{|c||}{ALS} & \multicolumn{3}{|c|}{ADF} \\ 
\hline  
$r$ & $\mean{res_C}$ & $\mean{res_P}$ & $\mean{time}$ & $\mean{res_C}$ & $\mean{res_P}$ & $\mean{time}$ \\ 
\hline \hline 
  2 & 8.1e-03(1.3) & 4.1e-03(1.2) & 0.1(0.0) & 8.1e-03(1.3) & 4.1e-03(1.2) & 0.1(0.0) \\ 
\hline \hline 
  3 & 1.9e-03(1.6) & 3.8e-04(1.1) & 0.6(0.1) & 1.9e-03(1.6) & 3.8e-04(1.1) & 0.4(0.1) \\ 
\hline \hline 
  4 & 5.8e-04(2.4) & 3.6e-05(1.2) & 9.0(2.7) & 5.7e-04(2.4) & 3.6e-05(1.2) & 4.9(1.1) \\ 
\hline \hline 
  5 & 4.6e-04(2.8) & 4.3e-06(1.2) & 80.4(27.3) & 5.0e-04(2.6) & 4.1e-06(1.2) & 50.1(17.6) \\ 
\hline \hline 
  6 & 3.0e-04(2.3) & 7.9e-07(1.4) & 260.6(72.2) & 2.9e-04(2.4) & 8.7e-07(1.2) & 131.3(29.9) \\ 
\hline \hline 
  7 & 1.6e-04(2.4) & 1.5e-07(1.3) & 761.9(124.4) & 1.5e-04(2.5) & 2.1e-07(1.2) & 284.1(44.9) \\ 
\hline \hline 
  8 & 1.8e-04(2.5) & 3.4e-08(1.4) & 1964.9(309.3) & 2.1e-04(2.4) & 8.1e-08(1.2) & 555.2(68.2) \\ 
\hline 
\end{tabular} 
}
\vspace{0.1cm}
\caption{\label{rtest} Convergence and timing with respect to the target rank $r=r_{final}$ for otherwise fixed parameters.}
\end{table}

In our third experiment we consider the variation of mode sizes $n \in \{6,12,24,48\}$. The results 
are given in Table \ref{ntest}. We observe a rather slow increase of the error which can be attributed 
to the random sampling.
\begin{table}[htb]
\centering
\resizebox{\columnwidth}{!}{%
\begin{tabular}{|c|c|c|c||c|c|c|} 
\hline 
\multicolumn{7}{|c|}{$d = 7$, $r = 3$ , $n$ varying, $C_{SD} =10$}\\ 
\hline 
& \multicolumn{3}{|c||}{ALS} & \multicolumn{3}{|c|}{ADF} \\ 
\hline  
$n$ & $\mean{res_C}$ & $\mean{res_P}$ & $\mean{time}$ & $\mean{res_C}$ & $\mean{res_P}$ & $\mean{time}$ \\ 
\hline \hline 
  6 & 9.2e-04(2.1) & 2.5e-04(1.2) & 0.2(0.1) & 9.2e-04(2.1) & 2.5e-04(1.2) & 0.1(0.1) \\ 
\hline \hline 
 12 & 1.9e-03(1.6) & 3.8e-04(1.1) & 0.6(0.1) & 1.9e-03(1.6) & 3.8e-04(1.1) & 0.4(0.1) \\ 
\hline \hline 
 24 & 3.4e-03(1.5) & 4.4e-04(1.1) & 1.5(0.4) & 3.4e-03(1.5) & 4.4e-04(1.1) & 1.0(0.3) \\ 
\hline \hline 
 48 & 3.9e-03(1.5) & 5.5e-04(1.1) & 4.5(1.3) & 3.9e-03(1.5) & 5.5e-04(1.1) & 3.3(1.0) \\ 
\hline 
\end{tabular} 
}
\vspace{0.1cm}
\caption{\label{ntest} Convergence and timing with respect to the mode sizes $n$ for otherwise fixed parameters.}
\end{table}

In our fourth and last experiment we vary the slice density $C_{SD} \in \{1,3,10,20,50\}$. 
The near best 
approximation, for $d = 7,\ n = 12,\ r = 3$, is obtained as in Remark \ref{exp_sum}. 
Its relative residual is $res_{\rm exp} = 1.34\cdot10^{-3}$.
The results in Table \ref{ctest} show that for $\csd\to\infty$, i.e. sampling more and more 
entries of the tensor, the reconstruction gets closer and closer to the best rank $r=3$ 
approximation of the tensor. Reasonably good results are already obtained for $\csd=3$. Note that the relative
residual on the sampling set is smaller than the optimal residual (due to overfitting).
\begin{table}[htb]
\centering
\resizebox{\columnwidth}{!}{%
\begin{tabular}{|c|c|c|c||c|c|c|} 
\hline 
\multicolumn{7}{|c|}{$d = 7 $ , $r = 3$, $n = 12$, $C_{SD}$ varying}\\ 
\hline 
& \multicolumn{3}{|c||}{ALS} & \multicolumn{3}{|c|}{ADF} \\ 
\hline  
$c$ & $\mean{res_C}$ & $\mean{res_P}$ & $\mean{time}$ & $\mean{res_C}$ & $\mean{res_P}$ & $\mean{time}$ \\ 
\hline \hline 
  1 & 4.7e-03(1.8) & 4.2e-05(1.3) & 1.8(0.4) & 4.2e-03(1.9) & 3.4e-05(1.3) & 1.6(0.3) \\ 
\hline \hline 
  3 & 2.8e-03(1.7) & 1.7e-04(1.2) & 0.7(0.2) & 2.8e-03(1.7) & 1.7e-04(1.2) & 0.4(0.1) \\ 
\hline \hline 
 10 & 1.9e-03(1.6) & 3.8e-04(1.1) & 0.6(0.1) & 1.9e-03(1.6) & 3.8e-04(1.1) & 0.4(0.1) \\ 
\hline \hline 
 20 & 2.0e-03(1.7) & 5.7e-04(1.2) & 0.8(0.2) & 2.0e-03(1.7) & 5.7e-04(1.2) & 0.5(0.1) \\ 
\hline \hline 
 50 & 1.4e-03(1.5) & 7.2e-04(1.1) & 1.1(0.1) & 1.4e-03(1.5) & 7.2e-04(1.1) & 0.8(0.1) \\ 
\hline 
\end{tabular} 
}
\vspace{0.1cm}
\caption{\label{ctest} Convergence and timing with respect to the slice density $C_{SD}$ for otherwise fixed parameters.}
\end{table}

The tensor $A$ from (\ref{int}) is not suitable for a high-dimensional high rank tensor completion
based on random samples, because the singular behavior is localized in one of the corners of the 
hypercube $[0,R]^d$. In order to better investigate the approximation quality of ALS and ADF, we 
consider the tensor $D \in \R^{\Ind}$ given by the entries
\[ D_{(i_1,\ldots,i_d)} := \left( 1 + \sum_{\mu = 1}^{d-1} \frac{i_{\mu}}{i_{\mu+1}} \right)^{-1}. \]
For all examples, we choose $d=7$, $n=15$, $C_{SD}=10$, $\varepsilon_{stop}=5\times 10^{-4}$ (ADF)
and $\varepsilon_{stop}=15\times 10^{-4}$ (ALS). 
\begin{figure}[htb]
  \begin{center}
    \begin{minipage}{\linewidth}
      \hspace*{-0.7cm}   
      
      \setlength\figureheight{4cm}
      \setlength\figurewidth{0.9\linewidth}
      \input{tgt2.tikz}
    \end{minipage}
  \end{center}
  \caption{\label{tgt2} ($d=7$, $r=6,8,10$, $n=15$, $C_{SD} = 10$) Plotted are, 
    for varying target ranks $r_{final} = 6,8,10$, the residual ${res_P}$ (dashed) as well as the
    control residual ${res_C}$ (continous) as functions of the total time (in seconds)
    for one trial, for ALS (black, upper curves) and ADF (blue, lower curves).}
\end{figure}
\begin{figure}[htb]
  \begin{center}
    \begin{minipage}{\linewidth}
      \hspace*{-0.7cm}   
      
      \setlength\figureheight{4cm}
      \setlength\figurewidth{0.9\linewidth}
      \input{tgt2_lr.tikz}
    \end{minipage}
  \end{center}
  \caption{\label{tgt2_lr} ($d=7$, $r=6,8,10$, $n=15$, $C_{SD} = 10$) Plotted are, 
    for varying target ranks $r_{final} = 6,8,10$, the residual ${res_P}$ (dashed) as well as the
    control residual ${res_C}$ (continous) as functions of the relative time $t_{rel}$ for one 
    trial, for ALS (black, upper curves) and ADF (blue, lower curves).    
    Both methods start with the same initial guess of rank $r_{final}-1$ obtained by ALS. 
  }
\end{figure}
In our first experiment in Figure \ref{tgt2} we compare the runtime for ALS and ADF to reach a target accuracy for a 
rank $r_{final}\in\{6,8,10\}$ approximation. 
In our second experiment in Figure \ref{tgt2_lr} we repeat the experiment from Figure \ref{tgt2} and try to 
exclude any 
effects due to different choices of stopping parameters or initial guesses. For this, we start both iterations 
with the same initial guess of rank $r_{final}-1$ obtained from ALS.
Instead of the total time $T$ we measure the relative time $t_{rel} := (T - T_1)/T_1$ with respect to the 
runtime $T_1$ for rank $r_{final}-1$ ALS. 
We observe that
the ADF algorithm is consistently faster than ALS. The reason for this is that both iterations
require a similar number of steps, but the complexity per step of ALS is inferior to that of ADF,
cf. Lemma \ref{re:comp} and Lemma \ref{pi}. These observations are highlighted 
in Figure \ref{time_to_iter}, where we display the average number of iterations required until 
the next rank increase and the average measured time per step for each rank (for $d=7$, $n=15$,
$r_{final}=14$, $C_{SD}=10$) of both ALS and ADF. 
\begin{figure}[htb]
  \begin{center}
    \begin{minipage}{\linewidth}
      \hspace*{-0.7cm}   
      
      \setlength\figureheight{4cm}
      \setlength\figurewidth{0.9\linewidth}
%
%
\begin{tikzpicture}

\begin{axis}[%
width=0.410625\figurewidth,
height=\figureheight,
at={(0.540296\figurewidth,0\figureheight)},
scale only axis,
xmin=0,
xmax=16,
xtick={ 1,  2,  3,  4,  5,  6,  7,  8,  9, 10, 11, 12, 13, 14, 15, 16, 17, 18, 19, 20},
xlabel={rank},
ymin=0,
ymax=50,
ylabel={average time per rank},
axis x line*=bottom,
axis y line*=left,
legend style={at={(0.03,0.97)},anchor=north west,legend cell align=left,align=left,draw=white!15!black}
]
\addplot [color=black,solid,line width=\plotlinewidth]
  table[row sep=crcr]{%
1	0.058403903325\\
2	0.0974393761714286\\
3	0.184159946214286\\
4	0.356111788380952\\
5	0.673505492530303\\
6	1.169333995\\
7	1.94173424579365\\
8	3.0960180749495\\
9	4.85229529808081\\
10	7.45036066441337\\
11	10.9792006439394\\
12	15.2707666922378\\
13	20.7528061561538\\
14	27.9652812103209\\
};
\addlegendentry{ALS};

\addplot [color=red,dotted,line width=\plotlinewidth]
  table[row sep=crcr]{%
1	0.0336535661744247\\
2	0.127612839030207\\
3	0.216321172547889\\
4	0.359281659022471\\
5	0.630070677593592\\
6	1.11633789424553\\
7	1.9198062618072\\
8	3.15627201995216\\
9	4.95560469519859\\
10	7.46174710090934\\
11	10.8327153372919\\
12	15.2405987913983\\
13	20.8715601371253\\
14	27.9258353352144\\
15	36.6177336332515\\
16	47.1756375656673\\
};
\addlegendentry{ALS polyfit degree 4};

\addplot [color=blue,solid,line width=\plotlinewidth]
  table[row sep=crcr]{%
1	0.060183066125\\
2	0.0747501969571428\\
3	0.0963586396428572\\
4	0.123941996176768\\
5	0.163751101409091\\
6	0.197180046142857\\
7	0.238415257373737\\
8	0.28222089041958\\
9	0.33072998974359\\
10	0.380704936615385\\
11	0.440565424384615\\
12	0.497381747027864\\
13	0.562478399867511\\
14	0.640913653205098\\
};
\addlegendentry{ADF};

\addplot [color=red,dotted,line width=\plotlinewidth]
  table[row sep=crcr]{%
1	0.0570257436373037\\
2	0.0764314910588927\\
3	0.100027600894779\\
4	0.127814073144963\\
5	0.159790907809444\\
6	0.195958104888223\\
7	0.236315664381298\\
8	0.280863586288672\\
9	0.329601870610342\\
10	0.38253051734631\\
11	0.439649526496576\\
12	0.500958898061138\\
13	0.566458632039998\\
14	0.636148728433156\\
15	0.710029187240611\\
16	0.788100008462363\\
};
\addlegendentry{ADF polyfit degree 2};

\end{axis}

\begin{axis}[%
width=0.410625\figurewidth,
height=\figureheight,
at={(0\figurewidth,0\figureheight)},
scale only axis,
xmin=0,
xmax=15,
xtick={ 1,  2,  3,  4,  5,  6,  7,  8,  9, 10, 11, 12, 13, 14, 15, 16, 17, 18, 19, 20},
xlabel={rank},
ymin=0,
ymax=25,
ylabel={average iterations per rank},
axis x line*=bottom,
axis y line*=left,
legend style={at={(0.03,0.97)},anchor=north west,legend cell align=left,align=left,draw=white!15!black}
]
\addplot [color=black,solid,line width=\plotlinewidth]
  table[row sep=crcr]{%
1	2\\
2	4\\
3	4\\
4	4.75\\
5	5.65\\
6	3\\
7	4.45\\
8	5.8\\
9	5.8\\
10	5.8\\
11	5.25\\
12	7.1\\
13	7.45\\
14	11.65\\
};
\addlegendentry{ALS};

\addplot [color=blue,solid,line width=\plotlinewidth]
  table[row sep=crcr]{%
1	2\\
2	4\\
3	4\\
4	5.15\\
5	6\\
6	4\\
7	5.05\\
8	6.4\\
9	7.3\\
10	7.2\\
11	7\\
12	9.4\\
13	14.35\\
14	23.95\\
};
\addlegendentry{ADF};

\end{axis}
\end{tikzpicture}%
    \end{minipage}
  \end{center}
  \caption{\label{time_to_iter} ($d = 7, n = 15, r_{final} = 14, C_{SD} = 10$) 
    Plotted are the average number of iterations 
    required until the next rank increase (left) and the average measured time per step for each rank (right) 
    for ALS (black) and ADF (blue).}
\end{figure}
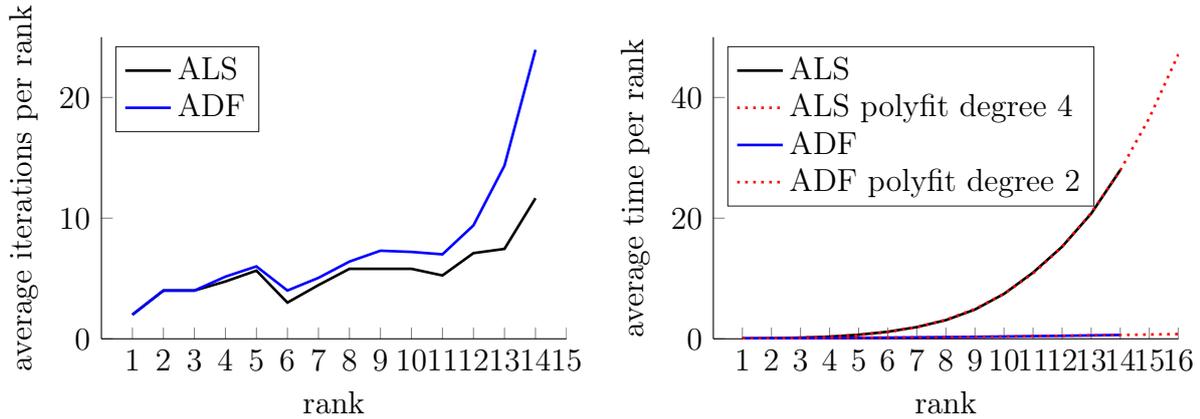
The detailed timing results of the experiments are given in Table \ref{rtest_AOG}, averaged over 
$20$ trials for each $r \in \{4,6,\ldots,14\}$.
\begin{table}[htb]
\centering
\resizebox{\columnwidth}{!}{%
\begin{tabular}{|c|c|c|c||c|c|c|}
\hline
\multicolumn{7}{|c|}{$d = 7$, $r$ varying, $n = 15$, $C_{SD} =10$}\\
\hline
& \multicolumn{3}{|c||}{ALS} & \multicolumn{3}{|c|}{ADF} \\
\hline
$r$ & $\mean{res_C}$ & $\mean{res_P}$ & $\mean{time}$ & $\mean{res_C}$ &
$\mean{res_P}$ & $\mean{time}$ \\
\hline \hline 
  4 & 3.5e-03(1.0) & 3.1e-03(1.0) & 0.6(0.2) & 3.5e-03(1.0) & 3.1e-03(1.0) & 0.3(0.1) \\ 
\hline \hline 
  6 & 9.6e-04(1.0) & 8.1e-04(1.0) & 6.8(2.4) & 9.6e-04(1.0) & 8.1e-04(1.0) & 1.4(0.0) \\ 
\hline \hline 
  8 & 1.9e-04(1.0) & 1.4e-04(1.0) & 64.7(2.5) & 1.9e-04(1.1) & 1.4e-04(1.0) & 5.2(0.2) \\ 
\hline \hline 
 10 & 6.3e-05(1.0) & 4.9e-05(1.0) & 133.4(5.0) & 6.3e-05(1.0) & 4.9e-05(1.0) & 15.2(0.5) \\ 
\hline \hline 
 12 & 2.4e-05(1.1) & 1.5e-05(1.0) & 466.7(15.8) & 2.4e-05(1.1) & 1.5e-05(1.0) & 34.9(0.8) \\ 
\hline \hline 
 14 & 7.8e-06(1.1) & 4.6e-06(1.0) & 1700.0(112.5) & 7.9e-06(1.1) & 4.6e-06(1.0) & 94.6(5.0) \\ 
\hline
\end{tabular}
}
\vspace{0.1cm}
\caption{\label{rtest_AOG} Convergence and timing with respect to the target rank $r=r_{final}$ for otherwise fixed parameters.}
\end{table}

\subsection{Reconstruction of a Low Rank Tensor without Noise}\label{sec:numrec}

As second group of examples, we consider quasi-random tensors with exact, common low 
TT ranks $A \in TT(r,\ldots,r)$ (cf. Definition \ref{TT tensor format}). Each quasi-random 
tensor is generated via a TT representation $A=A^G$ where we assign to each entry of each
block $G_1,\ldots,G_d$ a uniformly distributed random value in $[-0.5,0.5]$.
Each combination of parameters is tested $20$ times for different random $P$ and $C$
and stopping parameter 
$\varepsilon_{stop}:=5\times 10^{-4}$ (ADF) and
$\varepsilon_{stop}:=15\times 10^{-4}$ (ALS).
We consider 
such a reconstruction successful if $res_C < 10^{-6}$. 
First, we do not change the quasi-random tensor. In the test afterwards, we manipulate the singular 
values of the original quasi-random tensor.

\subsubsection{Quasi-random Tensors}\label{sec:num_rec1}

In the first test we consider the reconstruction of quasi-random tensors as described
above. Since the rank is exactly $r=1,\ldots,8$ it would in principle be possible to 
find a tensor of exactly rank $r$ that interpolates the sampled points. However, due to 
the nature of the random sampling and possible local minima we do not always reconstruct 
the tensor. 
The number of successful reconstructions for $20$ random tensors is displayed in $20$ 
shades of gray, from white $(0)$ to black $(\mbox{all } 20)$.
In Figure \ref{recplotd4} for $d=4$ and $d=5$ we observe 
that both ALS and ADF are able to reconstruct the tensor (with known target rank $r$) provided
that the slice density is high enough. For larger ranks $r$ it seems that a slice 
density of $\csd=4$ is enough, but for smaller ranks the slice density has to be larger in 
order to compensate for the randomness in both the tensor as well as the sampling set $P$.
\begin{figure}[htb]
  \begin{center}
    \begin{minipage}{\linewidth}

    \setlength\figureheight{2.5cm}
    \setlength\figurewidth{0.9\linewidth}
%
%
\definecolor{mycolor1}{rgb}{0.50000,0.50000,0.80000}%
\begin{tikzpicture}

\begin{axis}[%
width=0.192207\figurewidth,
height=\figureheight,
at={(0.758713\figurewidth,0\figureheight)},
scale only axis,
axis on top,
xmin=0.5,
xmax=8.5,
xtick={1,2,3,4,5,6,7,8},
xlabel={rank r},
ymin=0.5,
ymax=8.5,
ytick={1,2,3,4,5,6,7,8},
yticklabels={{2},{4},{8},{16},{32},{64},{128},{256}},
title={ADF (d = 5)}
]
\addplot [forget plot] graphics [xmin=0.5,xmax=8.5,ymin=0.5,ymax=8.5] {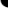};
\addplot [color=mycolor1,solid,forget plot]
  table[row sep=crcr]{%
1.5	0.5\\
1.5	8.5\\
};
\addplot [color=mycolor1,solid,forget plot]
  table[row sep=crcr]{%
0.5	1.5\\
8.5	1.5\\
};
\addplot [color=mycolor1,solid,forget plot]
  table[row sep=crcr]{%
2.5	0.5\\
2.5	8.5\\
};
\addplot [color=mycolor1,solid,forget plot]
  table[row sep=crcr]{%
0.5	2.5\\
8.5	2.5\\
};
\addplot [color=mycolor1,solid,forget plot]
  table[row sep=crcr]{%
3.5	0.5\\
3.5	8.5\\
};
\addplot [color=mycolor1,solid,forget plot]
  table[row sep=crcr]{%
0.5	3.5\\
8.5	3.5\\
};
\addplot [color=mycolor1,solid,forget plot]
  table[row sep=crcr]{%
4.5	0.5\\
4.5	8.5\\
};
\addplot [color=mycolor1,solid,forget plot]
  table[row sep=crcr]{%
0.5	4.5\\
8.5	4.5\\
};
\addplot [color=mycolor1,solid,forget plot]
  table[row sep=crcr]{%
5.5	0.5\\
5.5	8.5\\
};
\addplot [color=mycolor1,solid,forget plot]
  table[row sep=crcr]{%
0.5	5.5\\
8.5	5.5\\
};
\addplot [color=mycolor1,solid,forget plot]
  table[row sep=crcr]{%
6.5	0.5\\
6.5	8.5\\
};
\addplot [color=mycolor1,solid,forget plot]
  table[row sep=crcr]{%
0.5	6.5\\
8.5	6.5\\
};
\addplot [color=mycolor1,solid,forget plot]
  table[row sep=crcr]{%
7.5	0.5\\
7.5	8.5\\
};
\addplot [color=mycolor1,solid,forget plot]
  table[row sep=crcr]{%
0.5	7.5\\
8.5	7.5\\
};
\end{axis}

\begin{axis}[%
width=0.192207\figurewidth,
height=\figureheight,
at={(0.505809\figurewidth,0\figureheight)},
scale only axis,
axis on top,
xmin=0.5,
xmax=8.5,
xtick={1,2,3,4,5,6,7,8},
xlabel={rank r},
ymin=0.5,
ymax=8.5,
ytick={1,2,3,4,5,6,7,8},
yticklabels={{2},{4},{8},{16},{32},{64},{128},{256}},
title={ALS (d = 5)}
]
\addplot [forget plot] graphics [xmin=0.5,xmax=8.5,ymin=0.5,ymax=8.5] {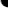};
\addplot [color=mycolor1,solid,forget plot]
  table[row sep=crcr]{%
1.5	0.5\\
1.5	8.5\\
};
\addplot [color=mycolor1,solid,forget plot]
  table[row sep=crcr]{%
0.5	1.5\\
8.5	1.5\\
};
\addplot [color=mycolor1,solid,forget plot]
  table[row sep=crcr]{%
2.5	0.5\\
2.5	8.5\\
};
\addplot [color=mycolor1,solid,forget plot]
  table[row sep=crcr]{%
0.5	2.5\\
8.5	2.5\\
};
\addplot [color=mycolor1,solid,forget plot]
  table[row sep=crcr]{%
3.5	0.5\\
3.5	8.5\\
};
\addplot [color=mycolor1,solid,forget plot]
  table[row sep=crcr]{%
0.5	3.5\\
8.5	3.5\\
};
\addplot [color=mycolor1,solid,forget plot]
  table[row sep=crcr]{%
4.5	0.5\\
4.5	8.5\\
};
\addplot [color=mycolor1,solid,forget plot]
  table[row sep=crcr]{%
0.5	4.5\\
8.5	4.5\\
};
\addplot [color=mycolor1,solid,forget plot]
  table[row sep=crcr]{%
5.5	0.5\\
5.5	8.5\\
};
\addplot [color=mycolor1,solid,forget plot]
  table[row sep=crcr]{%
0.5	5.5\\
8.5	5.5\\
};
\addplot [color=mycolor1,solid,forget plot]
  table[row sep=crcr]{%
6.5	0.5\\
6.5	8.5\\
};
\addplot [color=mycolor1,solid,forget plot]
  table[row sep=crcr]{%
0.5	6.5\\
8.5	6.5\\
};
\addplot [color=mycolor1,solid,forget plot]
  table[row sep=crcr]{%
7.5	0.5\\
7.5	8.5\\
};
\addplot [color=mycolor1,solid,forget plot]
  table[row sep=crcr]{%
0.5	7.5\\
8.5	7.5\\
};
\end{axis}

\begin{axis}[%
width=0.192207\figurewidth,
height=\figureheight,
at={(0.252904\figurewidth,0\figureheight)},
scale only axis,
axis on top,
xmin=0.5,
xmax=8.5,
xtick={1,2,3,4,5,6,7,8},
xlabel={rank r},
ymin=0.5,
ymax=8.5,
ytick={1,2,3,4,5,6,7,8},
yticklabels={{2},{4},{8},{16},{32},{64},{128},{256}},
title={ADF (d = 4)}
]
\addplot [forget plot] graphics [xmin=0.5,xmax=8.5,ymin=0.5,ymax=8.5] {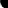};
\addplot [color=mycolor1,solid,forget plot]
  table[row sep=crcr]{%
1.5	0.5\\
1.5	8.5\\
};
\addplot [color=mycolor1,solid,forget plot]
  table[row sep=crcr]{%
0.5	1.5\\
8.5	1.5\\
};
\addplot [color=mycolor1,solid,forget plot]
  table[row sep=crcr]{%
2.5	0.5\\
2.5	8.5\\
};
\addplot [color=mycolor1,solid,forget plot]
  table[row sep=crcr]{%
0.5	2.5\\
8.5	2.5\\
};
\addplot [color=mycolor1,solid,forget plot]
  table[row sep=crcr]{%
3.5	0.5\\
3.5	8.5\\
};
\addplot [color=mycolor1,solid,forget plot]
  table[row sep=crcr]{%
0.5	3.5\\
8.5	3.5\\
};
\addplot [color=mycolor1,solid,forget plot]
  table[row sep=crcr]{%
4.5	0.5\\
4.5	8.5\\
};
\addplot [color=mycolor1,solid,forget plot]
  table[row sep=crcr]{%
0.5	4.5\\
8.5	4.5\\
};
\addplot [color=mycolor1,solid,forget plot]
  table[row sep=crcr]{%
5.5	0.5\\
5.5	8.5\\
};
\addplot [color=mycolor1,solid,forget plot]
  table[row sep=crcr]{%
0.5	5.5\\
8.5	5.5\\
};
\addplot [color=mycolor1,solid,forget plot]
  table[row sep=crcr]{%
6.5	0.5\\
6.5	8.5\\
};
\addplot [color=mycolor1,solid,forget plot]
  table[row sep=crcr]{%
0.5	6.5\\
8.5	6.5\\
};
\addplot [color=mycolor1,solid,forget plot]
  table[row sep=crcr]{%
7.5	0.5\\
7.5	8.5\\
};
\addplot [color=mycolor1,solid,forget plot]
  table[row sep=crcr]{%
0.5	7.5\\
8.5	7.5\\
};
\end{axis}

\begin{axis}[%
width=0.192207\figurewidth,
height=\figureheight,
at={(0\figurewidth,0\figureheight)},
scale only axis,
axis on top,
xmin=0.5,
xmax=8.5,
xtick={1,2,3,4,5,6,7,8},
xlabel={rank r},
ymin=0.5,
ymax=8.5,
ytick={1,2,3,4,5,6,7,8},
yticklabels={{2},{4},{8},{16},{32},{64},{128},{256}},
ylabel={$\text{slice density C}_{\text{SD}}$},
title={ALS (d = 4)}
]
\addplot [forget plot] graphics [xmin=0.5,xmax=8.5,ymin=0.5,ymax=8.5] {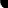};
\addplot [color=mycolor1,solid,forget plot]
  table[row sep=crcr]{%
1.5	0.5\\
1.5	8.5\\
};
\addplot [color=mycolor1,solid,forget plot]
  table[row sep=crcr]{%
0.5	1.5\\
8.5	1.5\\
};
\addplot [color=mycolor1,solid,forget plot]
  table[row sep=crcr]{%
2.5	0.5\\
2.5	8.5\\
};
\addplot [color=mycolor1,solid,forget plot]
  table[row sep=crcr]{%
0.5	2.5\\
8.5	2.5\\
};
\addplot [color=mycolor1,solid,forget plot]
  table[row sep=crcr]{%
3.5	0.5\\
3.5	8.5\\
};
\addplot [color=mycolor1,solid,forget plot]
  table[row sep=crcr]{%
0.5	3.5\\
8.5	3.5\\
};
\addplot [color=mycolor1,solid,forget plot]
  table[row sep=crcr]{%
4.5	0.5\\
4.5	8.5\\
};
\addplot [color=mycolor1,solid,forget plot]
  table[row sep=crcr]{%
0.5	4.5\\
8.5	4.5\\
};
\addplot [color=mycolor1,solid,forget plot]
  table[row sep=crcr]{%
5.5	0.5\\
5.5	8.5\\
};
\addplot [color=mycolor1,solid,forget plot]
  table[row sep=crcr]{%
0.5	5.5\\
8.5	5.5\\
};
\addplot [color=mycolor1,solid,forget plot]
  table[row sep=crcr]{%
6.5	0.5\\
6.5	8.5\\
};
\addplot [color=mycolor1,solid,forget plot]
  table[row sep=crcr]{%
0.5	6.5\\
8.5	6.5\\
};
\addplot [color=mycolor1,solid,forget plot]
  table[row sep=crcr]{%
7.5	0.5\\
7.5	8.5\\
};
\addplot [color=mycolor1,solid,forget plot]
  table[row sep=crcr]{%
0.5	7.5\\
8.5	7.5\\
};
\end{axis}
\end{tikzpicture}%
    \end{minipage}
  \end{center}
  \caption{\label{recplotd4} ($d=4,5$, $r$ varying, $n=12$, $C_{SD}$ varying, constant singular values) 
    Displayed as shades of 
    gray (white $(0)$ to black $(\mbox{all } 20)$) are the number of successful reconstructions for varying 
    target ranks $r = 1,\ldots,8$ and slice densities $C_{SD}=2,4,\ldots,256$ for ALS and ADF.}
\end{figure}

\subsubsection{Quasi-random Tensors with Decaying Singular Values}\label{sec:numrec2}

We base the second group of tests for random tensors on the same quasi-random tensors as above.
However, for each tensor and each matricization we enforce the
singular values to decay exponentially (that is $\sigma_i = 10^{-i}$) by rescaling them. 
We therefore
alternatingly adapt the singular values of the according matricizations 
of the random tensor.
Note that this can be done indirectly 
via the given representation of the random tensor.
The difference to the previous group of tests is that now the smaller singular values
are dominated by the large ones. The results in Figure \ref{recplotd4_decay} 
show a similar behaviour as before with the exception that, with respect to reconstruction 
capability, ADF performs slightly worse in $d=4$ and worse in dimension $d=5$.
 
\begin{figure}[htb]
  \begin{center}
    \begin{minipage}{\linewidth}
    \setlength\figureheight{2.5cm}
    \setlength\figurewidth{0.9\linewidth}
%
%
\definecolor{mycolor1}{rgb}{0.50000,0.50000,0.80000}%
\begin{tikzpicture}

\begin{axis}[%
width=0.192207\figurewidth,
height=\figureheight,
at={(0.505809\figurewidth,0\figureheight)},
scale only axis,
axis on top,
xmin=0.5,
xmax=8.5,
xtick={1,2,3,4,5,6,7,8},
xlabel={rank r},
ymin=0.5,
ymax=8.5,
ytick={1,2,3,4,5,6,7,8},
yticklabels={{2},{4},{8},{16},{32},{64},{128},{256}},
title={ALS (d = 5)}
]
\addplot [forget plot] graphics [xmin=0.5,xmax=8.5,ymin=0.5,ymax=8.5] {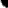};
\addplot [color=mycolor1,solid,forget plot]
  table[row sep=crcr]{%
1.5	0.5\\
1.5	8.5\\
};
\addplot [color=mycolor1,solid,forget plot]
  table[row sep=crcr]{%
0.5	1.5\\
8.5	1.5\\
};
\addplot [color=mycolor1,solid,forget plot]
  table[row sep=crcr]{%
2.5	0.5\\
2.5	8.5\\
};
\addplot [color=mycolor1,solid,forget plot]
  table[row sep=crcr]{%
0.5	2.5\\
8.5	2.5\\
};
\addplot [color=mycolor1,solid,forget plot]
  table[row sep=crcr]{%
3.5	0.5\\
3.5	8.5\\
};
\addplot [color=mycolor1,solid,forget plot]
  table[row sep=crcr]{%
0.5	3.5\\
8.5	3.5\\
};
\addplot [color=mycolor1,solid,forget plot]
  table[row sep=crcr]{%
4.5	0.5\\
4.5	8.5\\
};
\addplot [color=mycolor1,solid,forget plot]
  table[row sep=crcr]{%
0.5	4.5\\
8.5	4.5\\
};
\addplot [color=mycolor1,solid,forget plot]
  table[row sep=crcr]{%
5.5	0.5\\
5.5	8.5\\
};
\addplot [color=mycolor1,solid,forget plot]
  table[row sep=crcr]{%
0.5	5.5\\
8.5	5.5\\
};
\addplot [color=mycolor1,solid,forget plot]
  table[row sep=crcr]{%
6.5	0.5\\
6.5	8.5\\
};
\addplot [color=mycolor1,solid,forget plot]
  table[row sep=crcr]{%
0.5	6.5\\
8.5	6.5\\
};
\addplot [color=mycolor1,solid,forget plot]
  table[row sep=crcr]{%
7.5	0.5\\
7.5	8.5\\
};
\addplot [color=mycolor1,solid,forget plot]
  table[row sep=crcr]{%
0.5	7.5\\
8.5	7.5\\
};
\end{axis}

\begin{axis}[%
width=0.192207\figurewidth,
height=\figureheight,
at={(0.758713\figurewidth,0\figureheight)},
scale only axis,
axis on top,
xmin=0.5,
xmax=8.5,
xtick={1,2,3,4,5,6,7,8},
xlabel={rank r},
ymin=0.5,
ymax=8.5,
ytick={1,2,3,4,5,6,7,8},
yticklabels={{2},{4},{8},{16},{32},{64},{128},{256}},
title={ADF (d = 5)}
]
\addplot [forget plot] graphics [xmin=0.5,xmax=8.5,ymin=0.5,ymax=8.5] {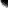};
\addplot [color=mycolor1,solid,forget plot]
  table[row sep=crcr]{%
1.5	0.5\\
1.5	8.5\\
};
\addplot [color=mycolor1,solid,forget plot]
  table[row sep=crcr]{%
0.5	1.5\\
8.5	1.5\\
};
\addplot [color=mycolor1,solid,forget plot]
  table[row sep=crcr]{%
2.5	0.5\\
2.5	8.5\\
};
\addplot [color=mycolor1,solid,forget plot]
  table[row sep=crcr]{%
0.5	2.5\\
8.5	2.5\\
};
\addplot [color=mycolor1,solid,forget plot]
  table[row sep=crcr]{%
3.5	0.5\\
3.5	8.5\\
};
\addplot [color=mycolor1,solid,forget plot]
  table[row sep=crcr]{%
0.5	3.5\\
8.5	3.5\\
};
\addplot [color=mycolor1,solid,forget plot]
  table[row sep=crcr]{%
4.5	0.5\\
4.5	8.5\\
};
\addplot [color=mycolor1,solid,forget plot]
  table[row sep=crcr]{%
0.5	4.5\\
8.5	4.5\\
};
\addplot [color=mycolor1,solid,forget plot]
  table[row sep=crcr]{%
5.5	0.5\\
5.5	8.5\\
};
\addplot [color=mycolor1,solid,forget plot]
  table[row sep=crcr]{%
0.5	5.5\\
8.5	5.5\\
};
\addplot [color=mycolor1,solid,forget plot]
  table[row sep=crcr]{%
6.5	0.5\\
6.5	8.5\\
};
\addplot [color=mycolor1,solid,forget plot]
  table[row sep=crcr]{%
0.5	6.5\\
8.5	6.5\\
};
\addplot [color=mycolor1,solid,forget plot]
  table[row sep=crcr]{%
7.5	0.5\\
7.5	8.5\\
};
\addplot [color=mycolor1,solid,forget plot]
  table[row sep=crcr]{%
0.5	7.5\\
8.5	7.5\\
};
\end{axis}

\begin{axis}[%
width=0.192207\figurewidth,
height=\figureheight,
at={(0.252904\figurewidth,0\figureheight)},
scale only axis,
axis on top,
xmin=0.5,
xmax=8.5,
xtick={1,2,3,4,5,6,7,8},
xlabel={rank r},
ymin=0.5,
ymax=8.5,
ytick={1,2,3,4,5,6,7,8},
yticklabels={{2},{4},{8},{16},{32},{64},{128},{256}},
title={ADF (d = 4)}
]
\addplot [forget plot] graphics [xmin=0.5,xmax=8.5,ymin=0.5,ymax=8.5] {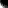};
\addplot [color=mycolor1,solid,forget plot]
  table[row sep=crcr]{%
1.5	0.5\\
1.5	8.5\\
};
\addplot [color=mycolor1,solid,forget plot]
  table[row sep=crcr]{%
0.5	1.5\\
8.5	1.5\\
};
\addplot [color=mycolor1,solid,forget plot]
  table[row sep=crcr]{%
2.5	0.5\\
2.5	8.5\\
};
\addplot [color=mycolor1,solid,forget plot]
  table[row sep=crcr]{%
0.5	2.5\\
8.5	2.5\\
};
\addplot [color=mycolor1,solid,forget plot]
  table[row sep=crcr]{%
3.5	0.5\\
3.5	8.5\\
};
\addplot [color=mycolor1,solid,forget plot]
  table[row sep=crcr]{%
0.5	3.5\\
8.5	3.5\\
};
\addplot [color=mycolor1,solid,forget plot]
  table[row sep=crcr]{%
4.5	0.5\\
4.5	8.5\\
};
\addplot [color=mycolor1,solid,forget plot]
  table[row sep=crcr]{%
0.5	4.5\\
8.5	4.5\\
};
\addplot [color=mycolor1,solid,forget plot]
  table[row sep=crcr]{%
5.5	0.5\\
5.5	8.5\\
};
\addplot [color=mycolor1,solid,forget plot]
  table[row sep=crcr]{%
0.5	5.5\\
8.5	5.5\\
};
\addplot [color=mycolor1,solid,forget plot]
  table[row sep=crcr]{%
6.5	0.5\\
6.5	8.5\\
};
\addplot [color=mycolor1,solid,forget plot]
  table[row sep=crcr]{%
0.5	6.5\\
8.5	6.5\\
};
\addplot [color=mycolor1,solid,forget plot]
  table[row sep=crcr]{%
7.5	0.5\\
7.5	8.5\\
};
\addplot [color=mycolor1,solid,forget plot]
  table[row sep=crcr]{%
0.5	7.5\\
8.5	7.5\\
};
\end{axis}

\begin{axis}[%
width=0.192207\figurewidth,
height=\figureheight,
at={(0\figurewidth,0\figureheight)},
scale only axis,
axis on top,
xmin=0.5,
xmax=8.5,
xtick={1,2,3,4,5,6,7,8},
xlabel={rank r},
ymin=0.5,
ymax=8.5,
ytick={1,2,3,4,5,6,7,8},
yticklabels={{2},{4},{8},{16},{32},{64},{128},{256}},
ylabel={$\text{slice density C}_{\text{SD}}$},
title={ALS (d = 4)}
]
\addplot [forget plot] graphics [xmin=0.5,xmax=8.5,ymin=0.5,ymax=8.5] {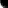};
\addplot [color=mycolor1,solid,forget plot]
  table[row sep=crcr]{%
1.5	0.5\\
1.5	8.5\\
};
\addplot [color=mycolor1,solid,forget plot]
  table[row sep=crcr]{%
0.5	1.5\\
8.5	1.5\\
};
\addplot [color=mycolor1,solid,forget plot]
  table[row sep=crcr]{%
2.5	0.5\\
2.5	8.5\\
};
\addplot [color=mycolor1,solid,forget plot]
  table[row sep=crcr]{%
0.5	2.5\\
8.5	2.5\\
};
\addplot [color=mycolor1,solid,forget plot]
  table[row sep=crcr]{%
3.5	0.5\\
3.5	8.5\\
};
\addplot [color=mycolor1,solid,forget plot]
  table[row sep=crcr]{%
0.5	3.5\\
8.5	3.5\\
};
\addplot [color=mycolor1,solid,forget plot]
  table[row sep=crcr]{%
4.5	0.5\\
4.5	8.5\\
};
\addplot [color=mycolor1,solid,forget plot]
  table[row sep=crcr]{%
0.5	4.5\\
8.5	4.5\\
};
\addplot [color=mycolor1,solid,forget plot]
  table[row sep=crcr]{%
5.5	0.5\\
5.5	8.5\\
};
\addplot [color=mycolor1,solid,forget plot]
  table[row sep=crcr]{%
0.5	5.5\\
8.5	5.5\\
};
\addplot [color=mycolor1,solid,forget plot]
  table[row sep=crcr]{%
6.5	0.5\\
6.5	8.5\\
};
\addplot [color=mycolor1,solid,forget plot]
  table[row sep=crcr]{%
0.5	6.5\\
8.5	6.5\\
};
\addplot [color=mycolor1,solid,forget plot]
  table[row sep=crcr]{%
7.5	0.5\\
7.5	8.5\\
};
\addplot [color=mycolor1,solid,forget plot]
  table[row sep=crcr]{%
0.5	7.5\\
8.5	7.5\\
};
\end{axis}
\end{tikzpicture}%
    \end{minipage}
  \end{center}
  \caption{\label{recplotd4_decay} ($d=4,5$, $r$ varying, $n=12$, $C_{SD}$ varying, decaying singular values) 
    Displayed as shades of 
    gray (white $(0)$ to black $(\mbox{all } 20)$) are the number of successful reconstructions for varying 
    target ranks $r = 1,\ldots,8$ and slice densities $C_{SD}=2,4,\ldots,256$ for ALS and ADF.}
\end{figure}

\subsubsection{Quasi-random Tensors with Decaying Singular Values and Gap}
\begin{figure}[htb]
  \begin{center}
    \begin{minipage}{\linewidth}
    \setlength\figureheight{5cm}
    \setlength\figurewidth{0.9\linewidth}
    \input{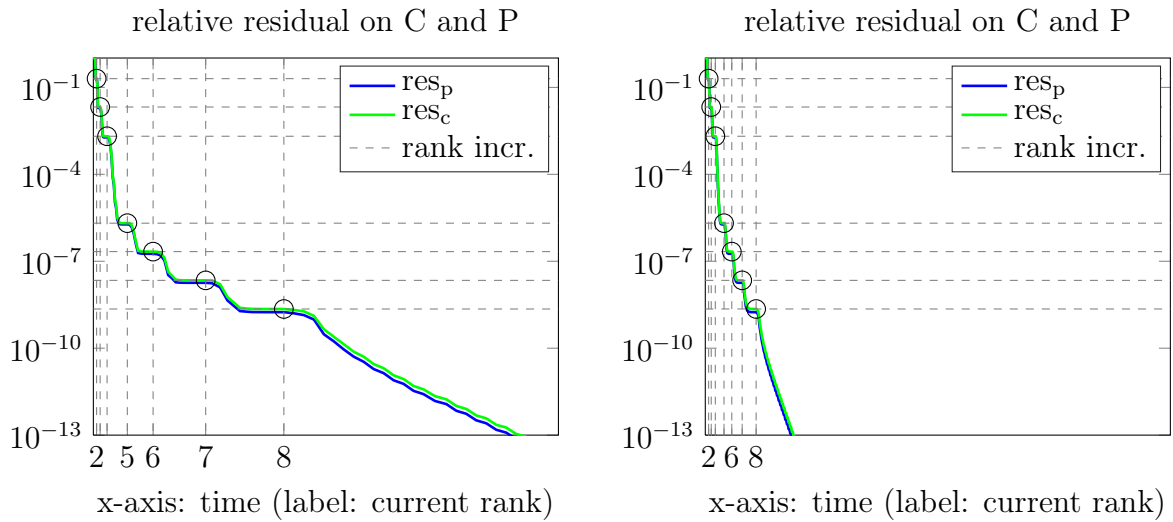}
    \end{minipage}
  \end{center}
  \caption{\label{gapplotr8} ($d=5$, $r=8$, $n=12$, $C_{SD}=10$) Plotted are the relative residuals for 
    one trial of a reconstruction of a tensor with a gap in its exponentially decaying singular values 
    for ALS (left) and ADF (right). 
    Additionally, circles and accordant, dashed lines as well as the x-axis label indicate when the 
    algorithm automatically increases ranks.}
\end{figure}
The third group of tests is again based on the quasi-random tensors of Subsection \ref{sec:numrec2}.
This time, the singular values of each matricization of each tensor are rescaled to 
$\sigma_i = 10^{-i}$ for $i \leq r/2$ and $\sigma_i = 10^{-i-2}$ for $i > r/2$, i.e., there is a gap
in the singular values after the first $r/2$ singular values.
We illustrate the results by two diagrams in Figure \ref{gapplotr8}, in which we plot the residuals 
$res_P$ and $res_C$ on the y-axis against the elapsed time on the x-axis. 
We fix the dimension $d=5$, $r=8$, the mode size $n=12$, and the slice density $C_{SD}=10$.
The dashed vertical and horizontal lines mark the points at which the rank is increased and are labelled
on the x-axis with the corresponding (higher) rank.

We observe that the gap in the singular values is clearly apparent in the approximation 
quality of the reconstruction, both in the given sample set $P$ as well as the control 
set $C$. Each of the residuals drops by three orders of magnitude if the rank approaches
$r/2$. Also, we can see that the residuals in $P$ and $C$ are almost the same, which is 
most likely a special property of random tensors. The comparison shows a clear advantage of
the ADF iteration over the ALS iteration with respect to timing.

\subsection{Reconstruction of a Low Rank Tensor with Noise}

In the fourth group of tests we repeat the ones from Subsection \ref{sec:numrec} 
but with perturbed tensors $\tilde{A} = A + 10^{-4} \nu \mathcal{E}$, where $A$ is 
generated as before and $\nu := \|A\|_P / \sqrt{\#P}$. The perturbation $\mathcal{E}$ is a 
tensor of the same proportions as $A$ and without any prescribed rank structure.
Each of its entries is assigned a uniform random value in $[-1,1]$. A test is considered 
successful if $res_C < 10^{-3}$, where the control set residual is evaluated for $A$ and
not $\tilde{A}$. However, no information about the non perturbed tensor is used in the 
algorithm. The results are identical to those of Subsection \ref{sec:numrec}, i.e. the 
perturbation has no influence on the reconstruction as long as the magnitude is below the
target accuracy. We do not yet have a theoretical justification for this very
pronounced effect and believe that a thorough analysis might reveal more insight.

\subsection{Stochastic Elliptic PDE with Karhunen-Lo\`eve Expansion}\label{subsection:kl}

Our last numerical example is a tensor completion problem based on an elliptic PDE with stochastic coefficient $a$,  
\begin{alignat*}{2}
 - \mbox{div} (a(x,y) \nabla u(x,y)) & = f(x), & \quad (x,y) & \in D \times \Theta, \\
 u(x,y) & = 0  & \quad  (x,y) & \in \partial D \times \Theta,
\end{alignat*}
where $y \in \Theta$ is a random variable and $D = [-1,1]$. The goal is to determine the 
expected value of the average of the solutions $\bar{u}(y):=\int_D u(x,y){\rm dx}$.
We follow the procedure described in \cite{ScGi11,KrStVa14} where first the stochastic
coefficient is replaced by a truncated $d+1$-term Karhunen-Lo\`eve (KL) expansion. Subsequently the 
solution space over the computational domain $D$ is discretised by finite elements
and the $d$ stochastic independent variables are sampled on a uniform grid, which yields 
averaged solutions $A_{i_1,\ldots,i_d} := \bar{u}(i_1,\ldots,i_d)$ depending on the parameters $i_{\mu}$. 
For each parameter combination $(i_1,\ldots,i_d)$, a deterministic problem has to be solved and the
average over all solutions gives the sought expected value. In this example, we choose $f(x) \equiv 1$
and use a finite element space with $m=50$ degrees of freedom.

In Figure \ref{pde_als1} we display the convergence for algebraically decaying KL eigenvalues 
$\sqrt{\lambda_{\mu}} =
(1+\mu)^{-2}$, final rank $r_{final} \in \{4,6,8\}$ for dimension 
$d = 5$, slice density $C_{SD} = 6$ and stopping parameter 
$\varepsilon_{stop}:=5\times 10^{-4}$ (ADF) and 
$\varepsilon_{stop}:=15\times 10^{-4}$ (ALS).
\begin{figure}[htb]
  \begin{center}
    \begin{minipage}{\linewidth}
      \hspace*{-0.7cm}   
      
      \setlength{\figureheight}{4cm}
      \setlength{\figurewidth}{0.9\linewidth}
      \input{tgt4.tikz}
    \end{minipage}
  \end{center}
  \caption{\label{pde_als1} ($d=5$, $r=4,6,8$, $n=100$, $C_{SD} = 6$) 
    Plotted are, for varying target ranks $r_{final} = 4,6,8$, 
    the residual ${res_P}$ (dashed) as well as the
    control residual ${res_C}$ (continous) as functions of the 
    total time (in seconds) for one trial,
    for ALS (black, upper curve) and ADF (blue, lower curve).}
\end{figure}
We observe that both methods eventually find a completed tensor of comparable
approximation quality, both in terms of the residual on $P$ and on $C$. 
The ADF iteration is consistently faster, and with increasing rank $r_{final}$ 
one can clearly see the advantage of the asymptotically lower complexity per 
step. 
\begin{figure}[htb]
  \begin{center}
    \begin{minipage}{\linewidth}
      \hspace*{-0.7cm}   
      
      \setlength{\figureheight}{4cm}
      \setlength{\figurewidth}{0.9\linewidth}
      \input{tgt4_lr.tikz}
    \end{minipage}
  \end{center}
  \caption{\label{pde_als2} ($d=5$, $r=4,6,8$, $n=100$, $C_{SD} = 6$) 
    Plotted are, for varying target ranks $r_{final} = 4,6,8$, 
    the residual ${res_P}$ (dashed) as well as the
    control residual ${res_C}$ (continous) as functions of the 
    relative time $t_{rel}$ for one trial,
    for ALS (black, upper curve) and ADF (blue, lower curve).
    Both methods start with the same initial guess of rank $r_{final}-1$
    obtained by ALS. 
  }
\end{figure}
In the tests in Figure \ref{pde_als2} we try to exclude any effects due to different choices 
of stopping parameters or initial guesses. For this, we start both iterations with the 
same initial guess of rank $r_{final}-1$ obtained from ALS.
Instead of the total time $T$ we measure the relative time $t_{rel} := (T - T_1)/T_1$ 
with respect to the runtime $T_1$ for rank $r_{final}-1$ ALS. 
Again, we observe that ADF is consistently faster.

\subsection{C Implementation}

The C implementation of the ALS and ADF algorithm, which was used for the latter results, 
can be found at
\begin{center}
{\tt http://www.igpm.rwth-aachen.de/personen/kraemer}
\end{center}

\section{Conclusions}\label{sec:conclusions}

In this article, we presented two variants of an alternating least squares algorithm that aim at 
finding a low tensor rank approximation to a tensor whose entries are known only in a small subset 
of all indices.
It is important to use a certain oversampling factor, respectively slice density $\csd$, 
in order to obtain a reasonable reconstruction of the tensor. In our numerical experiments 
it turns out that this factor depends on the dimension but can be decreased with increasing rank.
We obtain successful results already for the almost minimal value $\csd=2$. 
Both, the SOR-type solver ADF as well as the simple (and well known) alternating least squares 
method ALS are able to find 
reconstructions or approximations for moderate rank $r=1,\ldots,14$ and dimension $d=3,\ldots,55$. 
From our experiments we recommend to use the faster ADF algorithm, because the advantage of the 
${\cal O}(r^2d\#P)$ scaling over the
${\cal O}(r^4d\#P)$ scaling of ALS is already visible for rank $r=3$.
A modification or extension is necessary in order to treat varying TT ranks $r_1,\ldots,r_{d-1}$ 
instead of a uniform rank. Also, large mode sizes $n>100$ possibly require smoothness conditions and a 
refined sampling strategy.
The influence of noise on the reconstruction is rather harmless, where the noise can be unstructured
or of rank structure but of smaller magnitude than the desired target accuracy. It seems that the low
rank format introduces an automatic regularization in the same way as the singular value truncation
filters high frequency components.

 \bibliographystyle{plain}
 \bibliography{tensor}

\end{document}